\documentclass{amsart}

\usepackage{hyperref}
\usepackage{epsfig}
\usepackage{appendix}
\usepackage{amsfonts,amssymb,amsmath,amsthm,amscd}
\usepackage{latexsym}
\usepackage{mathtools}
\usepackage{graphicx}
\usepackage[all, cmtip]{xy}

\newtheorem{innercustomthm}{Theorem}
\newenvironment{thm}[1]
  {\renewcommand\theinnercustomthm{#1}\innercustomthm}
  {\endinnercustomthm}

\theoremstyle{plain}
		\newtheorem{theorem}{Theorem}[section]
		\newtheorem{lemma}[theorem]{Lemma}
		\newtheorem{corollary}[theorem]{Corollary}
		\newtheorem{proposition}[theorem]{Proposition}
		
		\newtheorem{conjecture}[theorem]{Conjecture}

		\newtheorem*{ntheorem}{Theorem}

		\newtheorem*{nclaim}{Claim}
		\newtheorem{fact}[theorem]{Fact}
\theoremstyle{definition}
		\newtheorem{definition}[theorem]{Definition}

\theoremstyle{remark}
		\newtheorem*{remark}{Remark}
		\newtheorem{example}[theorem]{Example}
		
\newcommand{\A}{{\mathbb{A}}}

\newcommand{\C}{{\mathbb{C}}}

\newcommand{\E}{{\mathbb{E}}}

\renewcommand{\H}{{\mathbb{H}}}

\newcommand{\K}{{\mathbb{K}}}

\renewcommand{\P}{{\mathbb{P}}}
\newcommand{\Q}{{\mathbb{Q}}}
\newcommand{\R}{{\mathbb{R}}}

\newcommand{\Z}{{\mathbb{Z}}}

\newcommand{\Cb}{{\mathbf{C}}}

\newcommand{\Fcal}{{\mathcal{F}}}

\newcommand{\Ocal}{{\mathcal{O}}}
\newcommand{\Pcal}{{\mathcal{P}}}

\newcommand{\Scal}{{\mathcal{S}}}

\renewcommand{\c}{\textup{c}}

\newcommand{\del}{\partial}

\newcommand{\id}{{\textup{id}}}

\DeclareMathOperator{\Diff}{Diff}

\DeclareMathOperator{\Hom}{Hom}

\DeclareMathOperator{\sheafhom}{\mathcal{H}\kern -.5pt \emph{om}}
\DeclareMathOperator{\Img}{Im}

\DeclareMathOperator{\Isom}{Isom}

\DeclareMathOperator{\length}{length}

\DeclareMathOperator{\PSL}{PSL}

\DeclareMathOperator{\Rep}{Rep}

\DeclareMathOperator{\SL}{SL}

\DeclareMathOperator{\SU}{SU}

\DeclareMathOperator{\Spec}{Spec}

\DeclareMathOperator{\vol}{vol}
\DeclareMathOperator{\sys}{sys}

\DeclareMathOperator{\tr}{tr}

\DeclareMathOperator{\Vol}{Vol}

\DeclareMathOperator{\II}{II}

\DeclareMathOperator{\Jac}{Jac}
\DeclareMathOperator{\dist}{dist}
\DeclareMathOperator{\Height}{Height}

\renewcommand{\inf}{\textup{inf}}

\newcommand{\git}{\mathbin{
  \mathchoice{/\mkern-6mu/}
    {/\mkern-6mu/}
    {/\mkern-5mu/}
    {/\mkern-5mu/}}}
		
\begin{document}
\title{Nonlinear descent on moduli of local systems}
\author{Junho Peter Whang}
\address{Dept of Mathematics, Massachusetts Institute of Technology, Cambridge MA}
\email{jwhang@mit.edu}
\date{\today}

\begin{abstract}
We establish a structure theorem for the integral points on moduli of special linear rank two local systems over surfaces, using mapping class group descent and boundedness results for systoles of local systems.
\end{abstract}

\maketitle

\setcounter{tocdepth}{1}
\tableofcontents

\section{Introduction} \label{sect:1}
\subsection{\unskip}  \label{sect:1.1}
This paper initiates our Diophantine study of moduli spaces for local systems on surfaces and their mapping class group dynamics. Let us fix a smooth compact oriented surface $\Sigma$ of genus $g\geq0$ with $n\geq0$ boundary curves, labeled $c_1,\dots,c_n$, satisfying $3g+n-3>0$. For each $k=(k_1,\dots,k_n)\in\C^n$, let $X_k$ denote the coarse moduli space of $\SL_2(\C)$-local systems on $\Sigma$ with trace $k_i$ along $c_i$. Each $X_k$ is an irreducible complex affine algebraic variety of dimension $6g+2n-6$, and we showed in \cite{whang} that it is log Calabi-Yau if the surface has nonempty boundary. For $k\in\Z^n$, the variety $X_k$ admits a natural model over $\Z$, its integral points corresponding to those local systems whose monodromy elements all have integer trace. The pure mapping class group of $\Sigma$ acts on $X_k$ via pullback of local systems, preserving the integral structure.

Our main Diophantine result for $X_k(\Z)$ is the following. Let us say that an algebraic variety $Z$ over $\C$ is \emph{parabolic} if each $\C$-point of $Z$ is in the image of some nonconstant morphism $\A^1\to Z$. A related notion is that of a log uniruled variety, defined in \cite{km}. Let us define a subvariety of $X_k$ to be \emph{degenerate} if it is contained in a parabolic subvariety of $X_k$, and \emph{nondegenerate} otherwise. These definitions and our formulation of Theorem \ref{theorem1} reflect considerations of Diophantine geometry for log Calabi-Yau varieties, further discussed in Section \ref{sect:1.3}.

\begin{theorem}
\label{theorem1}
The nondegenerate integral points in $X_k(\Z)$ consist of finitely many mapping class group orbits. There exists a parabolic proper closed subvariety of $X_k$ whose orbit gives precisely the locus of degenerate points on $X_k$.
\end{theorem}

In particular, $X_k(\Z)$ is generated from finitely many proper closed irreducible subvarieties of $X_k$ under the mapping class group. We shall obtain Theorem \ref{theorem1} by combining the results of this paper (outlined in Section \ref{sect:1.2}) with our work in \cite{whang4}, where the second part of Theorem \ref{theorem1} is proved together with a modular characterization of the degenerate points on $X_k$. The latter is summarized in the following paragraph.

Throughout this paper, by an \emph{essential curve} on $\Sigma$ we shall mean a simple closed curve on the interior of $\Sigma$ which cannot be continuously deformed into a point or a boundary component of $\Sigma$ (see Section \ref{sect:2.1} for our terminology). By \cite[Theorem 3.2]{whang4}, the degenerate points of $X_k$ correspond precisely to the local systems on $\Sigma$ which have central or parabolic monodromy (i.e.~have trace $\pm2$) along an essential curve or, if $(g,n,k)\neq(1,1,2)$, have parabolic (i.e.~reducible) restriction to some nontrivial ``pair of pants'' (a genus zero subsurface with three boundary curves each of which is either essential on $\Sigma$ or a boundary curve on $\Sigma$).

Diophantine problems on homogeneous varieties endowed with linear arithmetic group actions have been extensively studied; see \cite{sarnak}, \cite{drs} and references therein. A foundational result in this study is the finiteness of class numbers (i.e.~number of integral orbits), proved in general by Borel and Harish-Chandra \cite{bh} building on the works of Gauss, Hermite, Minkowski, Siegel, and others. Theorem \ref{theorem1} provides an analogue of this in our nonlinear and inhomogeneous setting. Ghosh and Sarnak \cite{gs} recently examined the behavior of the ``class numbers'' of $X_k$, in a slightly modified form, for varying $k$ in the case $(g,n)=(1,1)$ of a one holed torus.

When $(g,n)=(1,1)$, the variety $X_k$ is an affine algebraic surface with equation
\begin{align}
\label{eq:markoff}
x^2+y^2+z^2-xyz-2=k.
\end{align}
A Diophantine equation of this type and its nonlinear descent were first studied by Markoff \cite{markoff} in his 1880 work on binary quadratic forms, and versions have been frequently revisited \cite{hurwitz}, \cite{mordell}. (The Markoff uniqueness conjecture of Frobenius \cite{frob} states that a positive integral solution to (\ref{eq:markoff}) in the case $k=-2$ is determined up to permutation by its largest coordinate.) The asymptotic growth of (nondegenerate) integral solutions to Markoff-Hurwitz equations were analyzed in \cite{zagier} (see also \cite{mirzakhani}), \cite{baragar}, \cite{gmr}, and strong approximation for the Markoff equation was studied in \cite{bgs}, \cite{bgs1}. Recent works \cite{lm}, \cite{cwx} studied the Hasse principle and integral Brauer-Manin obstructions for the equations (\ref{eq:markoff}). These works inspire natural questions to be explored in the general context of moduli spaces for local systems over surfaces.

Theorem \ref{theorem1} also shares similarities with classical finite generation results for integral points on log Calabi-Yau varieties of linear type, such as algebraic tori and abelian varieties. In particular, the works of Vojta \cite{vojta2}, \cite{vojta3} and Faltings \cite{faltings} concerning integral points on subvarieties of semiabelian varieties motivate a similar Diophantine analysis of subvarieties on our moduli space $X_k$. As a first step in this direction, we show in \cite{whang4} that the integral points of any geometrically irreducible nondegenerate algebraic curve on $X_k$ can be effectively determined, and moreover satisfy a structure theorem.

\subsection{Main results}  \label{sect:1.2}
We now describe the contents of this paper and its main results leading up to Theorem \ref{theorem1}. Following Fermat, in this paper the term ``descent'' refers to a method of producing, from given (integral) solutions to an equation, new solutions of smaller height or magnitude, as a means to understand the structure of the full set of solutions. In our context, the solutions refer to points on the moduli space $X_k$ and the descent method is the mapping class group action. We emphasize that the moduli interpretation of $X_k$ plays a crucial role in our theory of descent, as it enables us to translate dynamical statements about the mapping class group action on $X_k$ into intrinsic statements on local systems, which are amenable to differential geometric tools.

In Section \ref{sect:2}, we collect relevant background on surfaces and moduli of local systems. The dynamical aspects of the mapping class group action on the complex points $X_k(\C)$ are not fully understood, but have been studied on certain special subloci. These include (but are not limited to) the locus of $\SU(2)$-local systems (see \cite{goldman}) and the Teichm\"uller locus parametrizing Fuchsian representations arising from marked hyperbolic structures on $\Sigma$ with geodesic boundary. This paper concerns the descent properties of the dynamics on $X_k(\C)$ beyond the classical setting.

In Section \ref{sect:3}, as a first step towards Theorem \ref{theorem1}, we introduce the notion of systole for $\SL_2$-local systems, and study its boundedness properties. The first main result of the section is Theorem \ref{sysarch}, concerning systoles of $\SL_2(\C)$-local systems, which we obtain using twisted harmonic maps from $\Sigma$ to the hyperbolic three-space $\H^3$ and systolic inequalities for Riemannian surfaces. A special case of the theorem is the following result. As before, an essential curve on $\Sigma$ is a simple closed curve which cannot be continuously deformed into a point or a boundary curve.

\begin{theorem}
\label{theorem2}
For each $B\geq0$, there is $A\geq0$ depending continuously on $B$ such that, given any representation $\rho:\pi_1\Sigma\to\SL_2(\C)$ whose boundary traces all have absolute value at most $B$, there is an essential curve $a\subset\Sigma$ with $|\tr\rho(a)|\leq A$.
\end{theorem}

In the special case of Fuchsian representations, this recovers Bers's boundedness theorem \cite{bers} for systoles of hyperbolic surfaces. When $\Sigma$ is a closed surface, Theorem \ref{theorem2} can also be deduced by combining Bers's theorem with a domination result of Deroin and Tholozan \cite{dt} which also uses harmonic maps. We remark in passing that an immediate byproduct of Theorem \ref{theorem2} is a weaker version of Theorem \ref{theorem1}: on each of our moduli spaces $X_k$, the set of integer points lies in the mapping class group orbit of a finite union of irreducible proper closed subvarieties.

Also in Section \ref{sect:3}, we use Bruhat-Tits trees and systolic inequalities to obtain an analogue (Theorem \ref{sysprop}) of Theorem \ref{sysarch} in a nonarchimedean setting. A notable special case is the following, which we mention as it plays a crucial role in \cite{whang4}.

\begin{theorem}
\label{theorem3}
Let $\Ocal$ be a discrete valuation ring with fraction field $F$. Given any representation $\rho:\pi_1\Sigma\to\SL_2(F)$ whose boundary traces all take values in $\Ocal$, there is an essential curve $a\subset\Sigma$ with $\tr\rho(a)\in\Ocal$.
\end{theorem}

Theorem \ref{theorem3} appears to have been known folklorically (certainly in the case of closed surfaces), and may be approached via Stallings's method involving surface maps to graphs. A higher-dimensional generalization of this approach can be found in Culler-Shalen theory \cite{cs} where, for example, group actions on trees are used to construct incompressible surfaces in three-manifolds. The systolic approach admits a different generalization to higher dimensions, which might be of independent interest (see remark at the end of Section \ref{sect:3.3.2} for details).

In Section \ref{sect:4}, we apply Theorem \ref{theorem2} and prove a compactness criterion for local systems, which extends Mumford's compactness criterion \cite{mumford} on moduli of closed Riemann surfaces as well as a related work of Bowditch, Maclachlan, and Reid \cite[Theorem 2.1]{bmr}. To state the result, it is useful to introduce the following notation. Let $X$ be the coarse moduli space of all $\SL_2(\C)$-local systems on $\Sigma$. Given arbitrary subsets $K\subseteq\C^n$ and $A\subseteq\C$, let $X_K(A)$ denote the set of local systems in $X(\C)$ whose boundary traces are in $K$ and whose trace along any essential curve on $\Sigma$ lies in $A$ (Lemma \ref{pointlem} shows that this agrees with the usual algebro-geometric notation when $A$ is a subring of $\C$ and $K$ is a singleton in $A^n$). Given subsets $A,B\subseteq \C$, let
$$\dist(A,B)=\inf\{|a-b|:a\in A,b\in B\}.$$
We shall endow the sets of complex points of algebraic varieties with the analytic topology. Let $\Gamma$ denote the pure mapping class group of the surface $\Sigma$.

\begin{theorem}
\label{theorem4}
Let $K\subset\C^n$ be a compact subset. Suppose that
\begin{enumerate}
	\item $A\subset\R$ and $\dist(A,\{\pm2\})>0$, or
	\item $A\subset\C$ and $\dist(A,[-2,2])>0$.
\end{enumerate}
There is then a compact subset $L\subseteq X_K(\C)$ satisfying
$X_K(A)\subseteq\Gamma\cdot L.$
\end{theorem}

The intuition behind the proof of Theorem \ref{theorem4} is the structure of an ``integrable system,'' associated to a pants decomposition of the surface, on the moduli spaces. This generalizes the familiar Fenchel-Nielsen coordinates on Teichm\"uller space, and was also notably used by Goldman \cite{goldman} to prove the ergodicity of mapping class group dynamics on moduli spaces for $\SU(2)$-local systems on surfaces. In practice, our proof of Theorem \ref{theorem4} follows an inductive argument on the topological complexity of the underlying surface (see also \cite{abo} which considers a type of ``geometric recursion'' in a different context).

When the set $A$ in Theorem \ref{theorem4} is moreover closed and discrete in $\R$ or $\C$, we deduce that
$X_K(A)$ consists of finitely many mapping class group orbits (Corollary \ref{finiteness}). This extends McKean's finiteness theorem \cite{mckean} for length isospectral families of closed Riemann surfaces. Specializing to $A = \Z\setminus\{\pm2\}$, we conclude that, for each $k\in\Z^n$, the set of
integral points on $X_k$ with no trace equal to $\pm2$ along essential curves of $\Sigma$ consists of finitely many mapping class group orbits. Since the set of nondegenerate points in $X_k(\Z)$ is contained in $X_k(\Z\setminus\{\pm2\})$ by the characterization given in Section \ref{sect:1.1}, this implies the first part of Theorem \ref{theorem1}. We note that a more economical proof of Theorem \ref{theorem1} is also given in Section \ref{sect:4.0}.

For each positive integer $d$, let $O_d$ denote the ring of integers in the imaginary quadratic ring $\Q(\sqrt{-d})$. A weak analogue of Theorem \ref{theorem1} can be proved for the set of all imaginary quadratic integral points on $X_k$ by considering $A=\bigcup_{d>0}O_d\setminus[-2,2]$ in Theorem \ref{theorem4}; in the case $(g,n)=(1,1)$, this was observed by Silverman \cite{silverman}. In particular, there is a proper closed subvariety $Z$ of $X_k$ with
$$\bigcup_{d>0}X_k(O_d)=\Gamma\cdot \bigcup_{d>0}Z(O_d)$$
where both sides are viewed as subsets of $X_k(\C)$.

Finally, in Section \ref{sect:5}, we collect further remarks on Theorem \ref{theorem1}. We briefly visit the theory of arithmetic hyperbolic surfaces in Section \ref{sect:5.1}, and derive elementary observations on the behavior of integral points on $X_k$. In Section \ref{sect:5.2}, we give an alternative proof (inspired by a suggestion of Curtis McMullen) of the finitude of mapping class group orbits for nondegenerate \emph{faithful} representations in $X_k(\Z)$ for surfaces with boundary.

\subsection{Diophantine remarks} \label{sect:1.3}
We now collect a number of Diophantine remarks to provide context and motivation for our work. Given a pair $(Z,D)$ consisting of a normal irreducible projective variety $Z$ and a reduced effective divisor $D$, its log canonical divisor $K_Z+D$ contains information about the global geometry of the open variety $V=Z\setminus D$. For instance, if $Z$ is a smooth irreducible projective curve and $D$ a finite collection of points, we have
$$\deg(K_Z+D)=-\chi(V(\C))$$
where the right hand side is the negative of the Euler characteristic of the Riemann surface $V(\C)$. Thus, in dimension one, the positivity of the log canonical divisor captures the negativity in curvature. The situation for higher dimensional varieties, while more complicated, follows a similar general pattern. For this reason, algebraic varieties with antiample, trivial (or torsion), and ample log canonical divisors form families whose Diophantine study of integral points deserves special focus.\footnote{In the projective setting, a guiding conjecture in birational geometry is that these classes of varieties form the ``building blocks'' of all projective varieties in a suitable sense.}

When $K_Z+D$ is \emph{positive} (e.g.~ample), a general expectation is that $V$ should have \emph{few} integral points. The celebrated theorems of Siegel and Faltings show that a smooth hyperbolic algebraic curve $V$, with an integral model over a number field, has at most finitely many integral ($=$ rational if $V$ is projective) points. While there are fewer general results in higher dimensions, several conjectures have been formulated in support of the expectation. One instance is the following.

\begin{conjecture}[Lang-Vojta]
\label{conjlv}
Suppose that $Z$ is a smooth irreducible projective variety with an integral model over a number field, $D$ an effective normal-crossing divisor on $Z$, and $V=Z\setminus D$. If $K_{Z}+D$ is (almost) ample, then the integral points of $V$ are contained in a proper Zariski closed subset of $V$.
\end{conjecture}

Note that the quantification of being ``few'' in higher dimensions is the failure to be Zariski dense. We remark that the Lang-Vojta conjecture is a consequence of a more general conjecture on Diophantine approximations, made by Vojta in analogy with Nevanlinna theory (see \cite{vojta} for details). For projective varieties, Lang \cite{lang} examined the notions of hyperbolicity in complex analytic geometry and put forth several conjectures such as the following, refining some features of the above:

\begin{conjecture}[Lang]
\label{conjl}
Suppose $V$ is a smooth irreducible projective variety of general type over a number field. The union $E$ of all non-constant rational images of $\P^1$ and abelian varieties in $V$ is a proper Zariski closed subset, and the complement of $V\setminus E$ has at most finitely many rational points over any number field.
\end{conjecture}

On the other hand, for pairs $(Z,D)$ with \emph{negative} (e.g.~antiample) log canonical divisors, an expectation is that $V$ should have \emph{many} integral points, at least after suitable extensions of number fields (potential density). In this setting, quantitative predictions such as the Batyrev-Manin conjecture have been made about the abundance of integral points; in some cases, the Hardy-Littlewood and related methods have been used to verify (or form the basis of) such predictions.

The above discussion suggests that the varieties with trivial log canonical divisor form the threshold or boundary case, where the expected Diophantine analysis of integral points is least well understood even conjecturally.

\begin{definition}
A normal quasiprojective variety $V$ is \emph{log Calabi-Yau} if it possesses a normal projective compactification $Z$ with canonical divisor $K_Z$ and reduced boundary divisor $D=Z\setminus V$ such that
$K_Z+D\sim 0$.
\end{definition}

\begin{remark}
In birational geometry, it is more customary to work with the pair $(Z, D)$, referred to as a \emph{(log) Calabi-Yau pair}. Often, one also imposes additional hypotheses on the singularity type of the pair $(Z,D)$, which we shall not do here.
\end{remark}

Examples of log Calabi-Yau varieties include algebraic tori, abelian varieties, projective K3 surfaces, and generic hypersurfaces of degree $n$ in $\A^n$. Algebraic tori and abelian varieties have been classically studied from the Diophantine perspective. Basic finite generation results for their integral points, such as the Dirichlet unit theorem and the Mordell-Weil theorem, reflect the richness of their \emph{dynamics} and provide bases for further arithmetic investigations. Other well-studied log Calabi-Yau (and more general) varieties include those among affine homogeneous varieties with linear arithmetic group actions, where the theorem of Borel--Harish-Chandra \cite{bh} provides the finitude of integral orbits.

The above results inspire us to study the behavior of integral points on other log Calabi-Yau varieties having rich dynamics. Let us make the following definition. Recall from Section \ref{sect:1.1} that an algebraic variety $Z$ over $\C$ is defined to be \emph{parabolic} if it is covered by images of nonconstant morphisms $\A^1\to Z$.

\begin{definition}
A subvariety of a log Calabi-Yau variety $V$ is \emph{degenerate} if it lies in a parabolic subvariety of $V$.
\end{definition}

Conjectures \ref{conjlv} and \ref{conjl} encourage us to ask, given a log Calabi-Yau variety $V$ with a model over $\Z$:
\begin{itemize}
	\item Is the set $V(\Z)$ generated from a proper Zariski closed subset?
	\item Is the set of degenerate points on $V$ generated from a parabolic subvariety?
	\item Is the set of nondegenerate points in $V(\Z)$ finitely generated? 
\end{itemize}
Here, to speak of ``generating'' points one assumes given some dynamical structure on $V$ compatible with the integral structure. While the above questions have affirmative answers for the log Calabi-Yau varieties of linear or homogeneous type mentioned above (and in fact such varieties have no parabolic subvarieties), for a general log Calabi-Yau variety this need not be the case (at least depending on the notion of \emph{dynamics} chosen). It therefore seems worthwhile to isolate the cases of positive answer as instances of \emph{arithmeticity}.

Let $\Sigma$ be a surface as in Section \ref{sect:1.1}, and suppose that the boundary of $\Sigma$ is nonempty. Let $X$ denote the moduli space of all $\SL_2(\C)$-local systems on $\Sigma$ (as in Section \ref{sect:1.2}), and for each $k\in\Z^n$ let $X_k\subset X$ be the subvariety given by prescribing the boundary traces of local systems (as in Section \ref{sect:1.1}). It is easy to see that $X$ is a parabolic variety. On the other hand, we showed in \cite{whang} that each $X_k$ is log Calabi-Yau, and the above discussion applies. Thus, Theorem \ref{theorem1} states that the mapping class group dynamics on each $X_k$ is arithmetic over the integers. We remark that, unlike linear arithmetic groups, the mapping class group of a
surface is defined with no \emph{a priori} connection to the integers, so it is surprising to find that such finiteness theorems as Theorem \ref{theorem1} should hold. Another point of note is that our method uses discreteness in $\C$ heavily, and hence is naturally restricted points of $X_k$ with values in $\Z$ or the integers of imaginary quadratic fields, as far as the analysis of integral points is concerned.

\subsection{Acknowledgements}  \label{sect:1.4}
This work was done as part of the author's Ph.D.~thesis at Princeton University. I thank my advisor Peter Sarnak and Phillip Griffiths for their guidance, unwavering encouragement, and generous sharing of insight. I also thank Sophie Morel for numerous helpful conversations during part of this work. I thank Curtis McMullen and Alan Reid for valuable suggestions and discussions, and I thank the referee for suggesting numerous improvements to the paper.

\section{Background}  \label{sect:2}
This section collects relevant background on surfaces and their moduli of local systems. It also serves to set the main notations and conventions in the paper.

\subsection{Surfaces}  \label{sect:2.1}
A \emph{surface} is an oriented two dimensional smooth manifold, which we assume to be compact with at most finitely many boundary components unless otherwise indicated. A connected surface is said to have type $(g,n)$ if it has genus $g$ and has $n$ boundary components. A \emph{curve} on a surface is an embedded copy of an unoriented circle, which we shall tacitly assume to be smooth in appropriate contexts. Given a surface $\Sigma$, we shall say that a curve $a\subset\Sigma$ is \emph{nondegenerate} if it does not bound a disk, and \emph{essential} if it is nondegenerate and is disjoint from, and not isotopic to, a boundary curve of $\Sigma$.

A \emph{multicurve} on $\Sigma$ is a finite union of disjoint curves on $\Sigma$. It is said to be \emph{nondegenerate}, resp.~\emph{essential}, if each of its components is. A \emph{pants decomposition} of $\Sigma$ is a maximal essential multicurve with pairwise nonisotopic components. If $\Sigma$ is a surface of type $(g,n)$ with $3g+n-3>0$, then any pants decomposition of $\Sigma$ consists of $3g+n-3$ essential curves. Given a surface $\Sigma$ and an essential multicurve $Q\subset\Sigma$, we denote by $\Sigma|Q$ the surface obtained by cutting $\Sigma$ along the curves in $Q$. An essential curve $a\subset\Sigma$ is \emph{separating} if the two boundary curves of $\Sigma|a$ corresponding to $a$ are on different connected components, and \emph{nonseparating} otherwise.

\subsubsection{Optimal sequence of generators} \label{sect:2.1.1}
Let $\Sigma$ be a surface of type $(g,n)$, and choose a base point $x\in\Sigma$. We have the \emph{standard presentation} of the fundamental group
\begin{equation}
\pi_1(\Sigma,x)=\langle \alpha_1,\beta_1',\dots,\alpha_g,\beta_g',\gamma_1,\dots,\gamma_n|[\alpha_1,\beta_1']\cdots[\alpha_g,\beta_g']\gamma_1\cdots \gamma_n\rangle
\label{eq:stdprs}
\end{equation}
where in particular $\gamma_1,\dots,\gamma_n$ correspond to loops around the boundary curves of $\Sigma$. For $i=1,\dots,g$, let $\beta_i$ be the based loop traversing $\beta_i'$ in the opposite direction. We can choose the sequence of generating loops $(\alpha_1,\beta_1,\dots,\alpha_g,\beta_g,\gamma_1,\dots,\gamma_n)$ so that it satisfies the following:
\begin{enumerate}
	\item[(1)]each loop in the sequence is simple,	
	\item[(2)]any two distinct loops in the sequence intersect exactly once (at $x$), and
	\item[(3)]every product of distinct elements in the sequence preserving the cyclic ordering can be represented by a simple loop in $\Sigma$.
\end{enumerate}
Some examples of products alluded to in (3) are $\alpha_1\beta_g$, $\alpha_1\alpha_2\beta_2\beta_g$, and $\beta_g\gamma_n\alpha_1$.

\begin{definition}
\label{optimal}
An \emph{optimal sequence of generators} for $\pi_1\Sigma$ is a sequence of loops $(a_1,b_1,\dots,a_g,b_g,c_1,\dots,c_n)$ obtained from a standard presentation of $\pi_1\Sigma$ as above.
\end{definition}

Figure \ref{fig1} shows an optimal sequence of generators for $(g,n)=(2,1)$.

\begin{figure}[ht]
    \centering
    \includegraphics{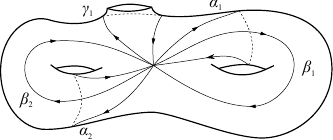}
    \caption{Optimal sequence of generators for $(g,n)=(2,1)$}
    \label{fig1}
\end{figure}

\subsubsection{Mapping class group} \label{sect:2.1.2}
Given a surface $\Sigma$, let $\Gamma=\Gamma(\Sigma)=\pi_0\Diff^+(\Sigma,\del\Sigma)$ denote its mapping class group. By definition, it is the group of isotopy classes of orientation preserving diffeomorphisms of $\Sigma$ fixing the boundary of $\Sigma$ pointwise. Given a (simple closed) curve $a\subset\Sigma$ disjoint from $\del\Sigma$, the associated (left) \emph{Dehn twist} $\tau_a\in\Gamma$ on $\Sigma$ is defined as follows. Let $S^1=\{z\in\C:|z|=1\}$ be the unit circle. Let $\tau$
be the diffeomorphism from $S^1\times[0,1]$ to itself given by $(z,t)\mapsto (ze^{2\pi i\xi(t)},t)$ where $\xi(t)$  is a smooth bump function of $t\in[0,1]$ that is 0 on a neighborhood of 0
and 1 on a neighborhood of 1. Choose a closed tubular neighborhood $N$ of $a$ in $\Sigma$, and an orientation preserving diffeomorphism $f:N\to S^1\times[0,1]$. The Dehn twist $\tau_a$ is given by
$$\tau_a(x)=\left\{\begin{array}{l l} f^{-1}\circ\tau\circ f(x)&\text{if $x\in N$},\\x &\text{otherwise.}\end{array}\right.$$
The class of $\tau_a$ in $\Gamma$ is independent of the choices involved above, and depends only on the isotopy class of $a$. It is a standard fact that $\Gamma=\Gamma(\Sigma)$ is generated by Dehn twists along simple closed curves in $\Sigma$ (see \cite[Chapter 4]{fm}).

\subsection{Character varieties}  \label{sect:2.2}
Throughout this paper, an \emph{algebraic variety} is a scheme of finite type over a field. Given an affine variety $V$ over a given field $k$, we denote by $k[V]$ its coordinate ring over $k$. If moreover $V$ is integral, then $k(V)$ denotes its function field over $k$. Given a commutative ring $A$ with unity, the elements of $A$ will be referred to as \emph{regular functions} on the affine scheme $\Spec A$.

\subsubsection{Character varieties of groups} \label{sect:2.2.1}

Let $\pi$ be a finitely generated group. Its ($\SL_2$) \emph{representation variety} $\Rep(\pi)$ is the affine scheme determined by the functor
$$A\mapsto\Hom(\pi,\SL_2(A))$$
for every commutative ring $A$. Given a sequence of generators of $\pi$ with $m$ elements, we have a presentation of $\Rep(\pi)$ as a closed subscheme of $\SL_2^m$ defined by equations coming from relations among the generators. For each $a\in\pi$, let $\tr_a$ be the regular function on $\Rep(\pi)$ given by $\rho\mapsto\tr\rho(a)$.

The ($\SL_2$) \emph{character variety} of $\pi$ over $\C$ is the affine invariant theoretic quotient
$$X(\pi)=\Rep(\pi)\git\SL_2=\Spec\C[\Rep(\pi)]^{\SL_2(\C)}$$
under the simultaneous conjugation action of $\SL_2$. Note that the regular function $\tr_{a}$ for each $a\in\pi$ descends to a regular function on $X(\pi)$. Moreover, $X(\pi)$ has a natural model over $\Z$, defined as the spectrum of
$$R(\pi)=\Z[\tr_a:a\in\pi]/(\tr_{1}-2,\tr_a\tr_b-\tr_{ab}-\tr_{ab^{-1}}).$$
The relations in the above presentation arise from the fact that the trace of the $2\times 2$ identity matrix is $2$, and $\tr(A)\tr(B)=\tr(AB)+\tr(AB^{-1})$ for every $A,B\in\SL_2(\C)$. Given an integral domain $A$ with fraction field $F$ of characteristic zero, the $A$-points of $X(\pi)$ parametrize the Jordan equivalence classes of $\SL_2(\overline F)$-representations of $\pi$ having character valued in $A$. We refer to \cite{horowitz}, \cite{ps}, \cite{saito} for details.

\begin{example}
\label{exfree}
We refer to Goldman \cite{goldman2} for details of examples below. Let $F_m$ denote the free group on $m\geq1$ generators $a_1,\dots,a_m$.
\begin{enumerate}
	\item[(1)] We have $\tr_{a_1}:X(F_1)\simeq\A^1$.
	\item[(2)] We have $(\tr_{a_1},\tr_{a_2},\tr_{a_1a_2}):X(F_2)\simeq\A^3$ by Fricke \cite[Section 2.2]{goldman2}.
	\item[(3)] The coordinate ring $\Q[X(F_3)]$ is the quotient of the polynomial ring
	$$\Q[\tr_{a_1},\tr_{a_2},\tr_{a_3},\tr_{a_1a_2},\tr_{a_2a_3},\tr_{a_1a_3},\tr_{a_1a_2a_3},\tr_{a_1a_3a_2}]$$
	by the ideal generated by two elements
$$
\tr_{a_1a_2a_3}+\tr_{a_1a_3a_2}-(\tr_{a_1a_2}\tr_{a_3}+\tr_{a_1a_3}\tr_{a_2}+\tr_{a_2a_3}\tr_{a_1}-\tr_{a_1}\tr_{a_2}\tr_{a_3})
$$
and
\begin{align*}
\tr_{a_1a_2a_3}\tr_{a_1a_3a_2}&-\{(\tr_{a_1}^2+\tr_{a_2}^2+\tr_{a_3}^2)+(\tr_{a_1a_2}^2+\tr_{a_2a_3}^2+\tr_{a_1a_3}^2)&\\
&\quad -(\tr_{a_1}\tr_{a_2}\tr_{a_1a_2}+\tr_{a_2}\tr_{a_3}\tr_{a_2a_3}+\tr_{a_1}\tr_{a_3}\tr_{a_1a_3})\\
&\quad +\tr_{a_1a_2}\tr_{a_2a_3}\tr_{a_1a_3}-4\}.
\end{align*}
\end{enumerate}
\end{example}

We record the following, which is attributed by Goldman \cite{goldman2} to Vogt \cite{vogt}.

\begin{lemma}
\label{rellem}
Given a finitely generated group $\pi$ and $a_1,a_2,a_3,a_4\in \pi$, we have
\begin{align*}
2{\tr_{a_1a_2a_3a_4}}&={\tr_{a_1}}{\tr_{a_2}}{\tr_{a_3}}{\tr_{a_4}}+{\tr_{a_1}}{\tr_{a_2a_3a_4}}+{\tr_{a_2}}{\tr_{a_3a_4a_1}}+{\tr_{a_3}}{\tr_{a_4a_1a_2}}\\
&\quad +{\tr_{a_4}}{\tr_{a_1a_2a_3}}+{\tr_{a_1a_2}}{\tr_{a_3a_4}}+{\tr_{a_4a_1}}{\tr_{a_2a_3}}-{\tr_{a_1a_3}}{\tr_{a_2a_4}}\\
&\quad -{\tr_{a_1}}{\tr_{a_2}}{\tr_{a_3a_4}}-{\tr_{a_3}}{\tr_{a_4}}{\tr_{a_1a_2}}-{\tr_{a_4}}{\tr_{a_1}}{\tr_{a_2a_3}}-{\tr_{a_2}}{\tr_{a_3}}{\tr_{a_4a_1}}.
\end{align*}
\end{lemma}

The above computation implies the following fact, which forms a special case of Procesi's theorem \cite{procesi} that ring of invariants of tuples of $N\times N$ matrices under simultaneous conjugation are (finitely) generated by the trace functions of products of matrices.

\begin{fact}
\label{fact}
If $\pi$ is a group generated by $a_1,\dots,a_m$, then $\Q[X(\pi)]$ is generated as a $\Q$-algebra by the collection $\{\tr_{a_{i_1}\cdots a_{i_k}}:1\leq i_1<\dots<i_k\leq m\}_{1\leq k\leq 3}$.
\end{fact}

\subsubsection{Moduli of local systems on manifolds} \label{sect:2.2.2}

Given a connected smooth (compact) manifold $M$, the moduli of local systems on $M$ we shall study is the character variety $X(M)=X(\pi_1 M)$ of its fundamental group. The complex points of $X(M)$ parametrize the Jordan equivalence classes of $\SL_2(\C)$-local systems on $M$. More generally, given a smooth manifold $M=M_1\sqcup\dots\sqcup M_m$ with finitely many connected components $M_i$, we define
$$X(M)=X(M_1)\times\dots\times X(M_m).$$
The construction of the moduli space $X(M)$ is functorial in the manifold $M$. Any smooth map $f:M\to N$ of manifolds induces a morphism $f^*:X(N)\to X(M)$, depending only on the homotopy class of $f$, given by pullback of local systems.

Let $\Sigma$ be a surface. For each curve $a\subset\Sigma$, there is a well-defined regular function $\tr_a:X(\Sigma)\to X(a)\simeq\A^1$, which agrees with $\tr_{\alpha}$ for any $\alpha\in\pi_1\Sigma$ freely homotopic to a parametrization of $a$. The boundary curves $\del\Sigma$ of $\Sigma$ induce a natural morphism
$$\tr_{\del\Sigma}=(-)|_{\del\Sigma}:X(\Sigma)\to X(\del\Sigma).$$
Now, writing $\Sigma=c_1\sqcup\dots\sqcup c_n$, we have an identification
$$X(\del\Sigma)=X(c_1)\times\dots\times X(c_n)\simeq\A^n$$
given by taking a local system on the disjoint union $\del\Sigma$ of $n$ circles to its sequence of traces along the curves. the morphism $\tr_{\del\Sigma}$ above may be viewed as an assignment to each $\rho\in X(\Sigma)$ its sequence of traces $\tr\rho(c_1),\dots,\tr\rho(c_n)$. The fibers of $\tr_{\del\Sigma}$
for $k\in\A^n$ will be denoted $X_k=X_k(\Sigma)$. Each $X_k$ is often referred to as a \emph{relative character variety} in the literature. If $\Sigma$ is a surface of type $(g,n)$ satisfying $3g+n-3>0$, the relative character variety $X_k(\Sigma)$ is an irreducible algebraic variety of dimension $6g+2n-6$.

We shall often simplify our notation by combining parentheses where applicable, e.g.~$X_k(\Sigma,\Z)=X_k(\Sigma)(\Z)$. Given a fixed surface $\Sigma$, a subset $K\subseteq X(\del\Sigma,\C)$, and a subset $A\subseteq\C$, we shall denote by
$$X_K(A)=X_K(\Sigma,A)$$
the set of all $\rho\in X(\Sigma,\C)$ such that $\tr_{\del\Sigma}(\rho)\in K$ and $\tr_a(\rho)\in A$ for every essential curve $a\subset\Sigma$. The following lemma shows that there is no risk of ambiguity with this notation.

\begin{lemma}
\label{pointlem}
If $A$ is a subring of $\C$ and $k\in A^n$, then $X_k$ has a model over $A$ and $X_k(A)$ recovers the set of $A$-valued points of $X_k$ in the sense of algebraic geometry.
\end{lemma}

\begin{proof}
Let $A$ and $k\in A^n$ be as above. We have a model of $X_k$ over $A$ with coordinate ring $\Spec R(\pi_1\Sigma)\otimes_\Z A$. It is clear that an $A$-valued point in the sense of algebraic geometry corresponds to a point in $X_k(A)$. The converse follows from the observation, using the identity $\tr_a\tr_b=\tr_{ab}+\tr_{ab^{-1}}$, that $\tr_b$ for every $b\in\pi_1\Sigma$ can be written as a $\Z$-linear combination of products of traces $\tr_a$ for nondegenerate curves $a\subset\Sigma$.
\end{proof}

Given an immersion $\Sigma'\to\Sigma$ of surfaces, we have the associated restriction
$$(-)|_{\Sigma'}:X(\Sigma)\to X(\Sigma').$$
The mapping class group $\Gamma(\Sigma)$ acts naturally on $X(\Sigma)$ by pullback of local systems, preserving the integral structure as well as each relative character variety $X_k(\Sigma)$ and the sets $X_K(\Sigma,A)$ defined above. As mentioned in Section \ref{sect:1}, this paper is largely concerned with the descent properties of the dynamics on the complex points of $X(\Sigma)$ beyond the classical setting, e.g.~of Teichm\"uller spaces.

\subsubsection{Reconstruction and lifts of Dehn twists} \label{sect:2.2.3}
Let $\Sigma$ be a surface of type $(g,n)$ with $3g+n-3>0$, and let $a\subset\Sigma$ be an essential curve. Let $x\in\Sigma$ be a base point lying on $a$, and let $\alpha$ be a simple based loop parametrizing $a$. We shall describe the reconstruction of representations $\rho:\pi_1(\Sigma,x)\to\SL_2(\C)$ from representations on connected components of $\Sigma|a$, as well as associated lifts of Dehn twists. Our main reference is Goldman-Xia \cite{gx}. There are two cases to consider, according to whether $a$ is separating or nonseparating.

{\bf Nonseparating curves.}
Suppose that $a$ is nonseparating, so $\Sigma|a$ is connected. Let $a_1$ and $a_2$ be the boundary curves of $\Sigma|a$ corresponding to $a$, and let $(x_i,\alpha_i)$ be the lifts of $(x,\alpha)$ to each $a_i$. We shall assume that we have chosen the numberings so that the interior of $\Sigma|a$ lies to the left as one travels along $\alpha_1$. Let $\beta$ be a simple loop on $\Sigma$ based at $x$, intersecting the curve $a$ once transversely at the base point, such that $\beta$ lifts to a path $\beta'$ in $\Sigma|a$ from $x_2$ to $x_1$. Let us denote by $\alpha_2'$ the loop based at $x_1$ given by the path $\alpha_2'=(\beta')^{-1}\alpha_2\beta'$,
where $(\beta')^{-1}$ refers to the path $\beta'$ traversed in the opposite direction. The immersion $\Sigma|a\to\Sigma$ induces an embedding $\pi_1(\Sigma|a,x_1)\to\pi_1(\Sigma,x)$, giving us the isomorphism
$$\pi_1(\Sigma,x)=(\pi_1(\Sigma|a,x_1)\vee\langle\beta\rangle)/(\alpha_2'=\beta^{-1}\alpha_1\beta).$$
Thus, any representation $\rho:\pi_1(\Sigma,x)\to\SL_2(\C)$ is determined uniquely by a pair $(\rho',B)$, where $\rho':\pi_1(\Sigma|a,x_1)\to\SL_2(\C)$ is a representation and $B\in\SL_2(\C)$ is an element such that $\rho'(\alpha_2')=B^{-1}\rho'(\alpha_1)B$, with the correspondence
$$\rho\mapsto(\rho',B)=(\rho|_{\pi_1(\Sigma|a,x_1)},\rho(\beta)).$$
We define an automorphism $\tau_\alpha$ of $\Hom(\pi_1(\Sigma,x),\SL_2)$ as follows. Given $\rho=(\rho',B)$, we set $\tau_\alpha(\rho',B)=(\rho',B')$ where $B'=\rho(\alpha)B$. This descends to the action $\tau_a$ of the left Dehn twist action along $a$ on the moduli space $X(\Sigma)$.

{\bf Separating curves.}
Suppose that $a$ is separating, so we have $\Sigma|a=\Sigma_1\sqcup\Sigma_2$ with each $\Sigma_i$ of type $(g_i,n_i)$ satisfying $2g_i+n_i-2>0$. Let $a_i$ be the boundary curve of $\Sigma_i$ corresponding to $a$. Let $(x_i,\alpha_i)$ be the lift of $(x,\alpha)$ to $a_i$. We shall assume that we have chosen the numberings so that the interior of $\Sigma_1$ lies to the left as one travels along $\alpha_1$. The immersions $\Sigma_i\hookrightarrow\Sigma$ of the surfaces induce embeddings
$\pi_1(\Sigma_i,x_i)\to\pi_1(\Sigma,x)$
of fundamental groups, and we have an isomorphism
$$\pi_1(\Sigma,x)\simeq(\pi_1(\Sigma_1,x_1)\vee\pi_1(\Sigma_2,x_2))/(\alpha_1=\alpha_2).$$
Thus, any representation
$\rho:\pi_1(\Sigma,x)\to\SL_2(\C)$
is determined uniquely by a pair $(\rho_1,\rho_2)$ of representations $\rho_i:\pi_1(\Sigma_i,x_i)\to\SL_2(\C)$ such that $\rho_1(\alpha_1)=\rho_2(\alpha_2)$, with the correspondence
$$\rho\mapsto(\rho_1,\rho_2)=(\rho|_{\pi_1(\Sigma_1,x_1)},\rho|_{\pi_1(\Sigma_2,x_2)}).$$
We define an automorphism $\tau_{\alpha}$ of $\Hom(\pi_1(\Sigma,x),\SL_2)$ as follows. For a representation $\rho=(\rho_1,\rho_2)$, we set $\tau_\alpha(\rho_1,\rho_2)=(\rho_1,\rho_2')$
where $$\rho_2'(\gamma)=\rho(\alpha)\rho_2(\gamma)\rho(\alpha)^{-1}$$
for every $\gamma\in\pi_1(\Sigma_2,x_2)$. This descends to the action $\tau_a$ of the left Dehn twist along $a$ on the moduli space $X(\Sigma)$.

\subsection{Markoff type cubic surfaces}  \label{sect:2.3}
Here, we give a description of the moduli spaces $X_k(\Sigma)$ and their mapping class group dynamics for $(g,n)=(1,1)$ and $(0,4)$. These cases are distinguished by the fact that each $X_k$ is an affine cubic algebraic surface with an explicit equation.

\subsubsection{Case $(g,n)=(1,1)$} \label{sect:2.3.1}
Let $\Sigma$ be a surface of type $(g,n)=(1,1)$, i.e.~a one holed torus. Let $(\alpha,\beta,\gamma)$ be an optimal sequence of generators for $\pi_1\Sigma$ (as in Definition \ref{optimal}). By Example \ref{exfree}.(2), we have an identification
$(\tr_\alpha,\tr_\beta,\tr_{\alpha\beta}):X(\Sigma)\simeq\A^3$.
From the trace relations in Section \ref{sect:2.2}, we obtain
\begin{align*}
\tr_\gamma&=\tr_{\alpha\beta\alpha^{-1}\beta^{-1}}=\tr_{\alpha\beta\alpha^{-1}}\tr_{\beta^{-1}}-\tr_{\alpha\beta\alpha^{-1}\beta}\\
&=\tr_{\beta}^2-\tr_{\alpha\beta}\tr_{\alpha^{-1}\beta}+\tr_{\alpha\alpha}=\tr_{\beta}^2-\tr_{\alpha\beta}(\tr_{\alpha^{-1}}\tr_{\beta}-\tr_{\alpha\beta})+\tr_{\alpha}^2-\tr_{1}\\
&=\tr_{\alpha}^2+\tr_{\beta}^2+\tr_{\alpha\beta}^2-\tr_{\alpha}\tr_{\beta}\tr_{\alpha\beta}-2.
\end{align*}
Writing $(x,y,z)=(\tr_{\alpha},\tr_{\beta},\tr_{\alpha\beta})$ so that each of the variables $x$, $y$, and $z$ corresponds to an essential curve on $\Sigma$ as depicted in Figure \ref{fig2}, the moduli space $X_k\subset X$ has an explicit presentation as an affine cubic algebraic surface in $\A_{x,y,z}^3$ with equation
$$x^2+y^2+z^2-xyz-2=k.$$
The mapping class group $\Gamma=\Gamma(\Sigma)$ acts on $X_k$ via polynomial transformations. For convenience, let $a$, $b$, and $ab$ denote the essential curves lying in the free homotopy classes of loops $\alpha$, $\beta$, and $\alpha\beta$, respectively. We have the following descriptions of the associated Dehn twist actions.

\begin{lemma}
\label{dehnt}
The Dehn twist actions $\tau_a$, $\tau_b$, and $\tau_{ab}$ on $X(\Sigma)$ are given by
\begin{align*}
\tau_a^*&:(x,y,z)\mapsto(x,z,xz-y),\\
\tau_b^*&:(x,y,z)\mapsto(xy-z,y,x),\\
\tau_{ab}^*&:(x,y,z)\mapsto(y,yz-x,z)
\end{align*}
in terms of the above coordinates.
\end{lemma}

\begin{proof}
Note that $\tau_a(a)$ has the homotopy class of $\alpha$, $\tau_a(b)$ has the homotopy class of $\alpha\beta$, and $\tau_a(ab)$ has the homotopy class of $\alpha\alpha\beta$. Noting that $\tr_{\alpha\alpha\beta}=\tr_{\alpha}\tr_{\alpha\beta}-\tr_{\beta}$, we obtain the desired expression for $\tau_a^*$. The other Dehn twists are similar.
\end{proof}

Each of the morphisms $\tau_a^*$, $\tau_b^*$, and $\tau_{ab}^*$ is a composition of a Vieta involution, i.e.~morphism of the form
$(x,y,z)\mapsto(x,y,xy-z)$,
with a transposition of two coordinates.

\begin{figure}[ht]
    \includegraphics{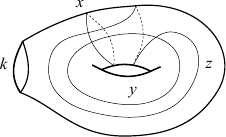}
    \caption{Curves on a surface of type $(1,1)$ with corresponding functions}
    \label{fig2}
\end{figure}

\subsubsection{Case $(g,n)=(0,4)$} \label{sect:2.3.2}
Let $\Sigma$ be a surface of type $(0,4)$, i.e.~a four holed sphere. Let $(\gamma_1,\dots,\gamma_4)$ be an optimal sequence of generators for $\pi_1\Sigma$ (as in Definition \ref{optimal}). Let us set
$$(x,y,z)=(\tr_{\gamma_1\gamma_2},\tr_{\gamma_2\gamma_3},\tr_{\gamma_1\gamma_3}),$$
so that each of the variables corresponds to an essential curve on $\Sigma$ as depicted in Figure \ref{fig3}. By Example \ref{exfree}.(3), for $k=(k_1,k_2,k_3,k_4)\in\A^4(
\C)$ the relative character variety $X_k=X_k(\Sigma)$ is an affine cubic algebraic surface in $\A_{x,y,z}^3$ given by the equation
$$x^2+y^2+z^2+xyz=ax+by+cz+d$$
with
$$\left\{\begin{array}{l}a=k_1k_2+k_3k_4\\b=k_1k_4+k_2k_3\\c=k_1k_3+k_2k_4\end{array}\right.\quad\text{and}\quad d=4-\sum_{i=1}^4k_i^2-\prod_{i=1}^4k_i.$$
The mapping class group $\Gamma=\Gamma(\Sigma)$ acts on $X_k$ via polynomial transformations. In particular, by an elementary argument as in Lemma \ref{dehnt}, we see that Dehn twists give rise to morphisms
\begin{align*}
&(x,y,z)\mapsto(x,b-xz-y,c-x(b-xz-y)-z),\\
&(x,y,z)\mapsto(a-y(c-xy-z)-x,y,c-xy-z),\\
&(x,y,z)\mapsto(a-yz-x,b-(a-yz-x)z-y,z).
\end{align*}
Note that each transformation is a composition of two of the three Vieta involutions defined on $X_k$:
\begin{align*}
&(x,y,z)\mapsto(a-yz-x,y,z),\\
&(x,y,z)\mapsto(x,b-xz-y,z),\\
&(x,y,z)\mapsto(x,y,c-xy-z).
\end{align*}
\begin{figure}[ht]
    \centering
    \includegraphics{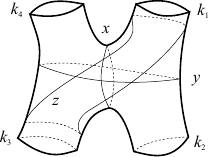}
    \caption{Curves on surfaces of type $(0,4)$ with corresponding functions}
    \label{fig3}
\end{figure}

\section{Systoles of local systems}  \label{sect:3}
In this section, we derive boundedness results for systoles of $\SL_2$-local systems on surfaces. In Section \ref{sect:3.1}, we briefly recall relevant notions and results in differential geometry, mainly to set up our notations and conventions. In Section \ref{sect:3.2}, we discuss existence results for harmonic maps needed in our anlaysis. In Section \ref{sect:3.3}, we prove the main results of this section, from which we derive Theorems \ref{theorem2} and \ref{theorem3} as corollaries in Section \ref{sect:3.3.1}.

\subsection{Differential geometry}  \label{sect:3.1}
\subsubsection{Curvature} \label{sect:3.1.1}
Let $M$ be a smooth manifold with Riemannian metric $\sigma=\langle-,-\rangle$. Let $\nabla$ be the associated Levi-Civita connection, and $R=-\nabla^2$ the Riemannian curvature tensor. For $p\in M$ and two-dimensional subspace $V\subseteq T_pM$, let $K(V)$ denote the sectional curvature of $V$. For any real $\epsilon>0$, we shall denote by $M(\epsilon)$ the smooth manifold $M$ with the scaled metric $\epsilon\sigma=\epsilon\langle-,-\rangle$.

When discussing a geometric construction in the presence of multiple manifolds, we shall use a subscript or superscript to indicate to which manifold it belongs. For instance, we shall write $K^M$ for the sectional curvature on $M$. Given smooth manifolds $M$ and $N$, let $\pi_M:M\times N\to M$ and $\pi_N:M\times N\to N$ be the projections. We record a basic lemma about sectional curvature on product manifolds.

\begin{lemma}
\label{scurv}
Let $M$ and $N$ be Riemannian manifolds of dimension at least $2$.
\begin{enumerate}
	\item[\textup{(1)}] If $M$ and $N$ have nonpositive curvature, then so does $M\times N$.
	\item[\textup{(2)}] For every $(p,q)\in M\times N$ and two-dimensional subspace $V\subseteq T_{(p,q)}(M\times N)$ such that $\dim_\R d\pi_N(V)=2$,	$$\lim_{\epsilon\to0}K^{M(\epsilon)\times N}(V)=K^N(d\pi_N(V)).$$
\end{enumerate}
\end{lemma}

\begin{proof}
A vector field on $T(M\times N)\simeq\pi_M^*TM\oplus\pi_N^*TN$ can be written as $(X,Y)$ for some vector field $X$ on $M$ and $Y$ on $N$. With this identification, the Levi-Civita connection on $M\times N$ is given by
$$\nabla_{(X_1,Y_1)}^{M\times N}(X_2,Y_2)=(\nabla_{X_1}^MX_2,\nabla_{Y_1}^NY_2)$$
for vector fields $X_i$ on $M$ and $Y_i$ on $N$. Therefore, by definition we have
\begin{align*}
&R^{M\times N}((X_1,Y_1),(X_2,Y_2))(X_3,Y_3)\\
&=(-\nabla_{(X_1,Y_1)}\nabla_{(X_2,Y_2)}+\nabla_{(X_2,Y_2)}\nabla_{(X_1,Y_1)}+\nabla_{[(X_1,Y_1),(X_2,Y_2)]})(X_3,Y_3)\\
&=(R^M(X_1,X_2)X_3,R^M(Y_1,Y_2)Y_3).
\end{align*}
It follows that, given a point $(p,q)\in M\times N$ and the two-dimensional subspace $V\subseteq T_{(p,q)}(M\times N)$ generated by two vectors $(x_1,y_1)$ and $(x_2,y_2)$, the sectional curvature $K^{M\times N}(V)$ of $V$ is therefore given by
$$\frac{\langle R^M(x_1,x_2)x_1,x_2\rangle_M +\langle R^N(y_1,y_2)y_1,y_2\rangle_N}{(\langle x_1,x_1\rangle_M+\langle y_1,y_1\rangle_N)(\langle x_2,x_2\rangle_M+\langle y_2,y_2\rangle_N)-(\langle x_1,x_2\rangle_M+\langle y_1,y_2\rangle_N)^2}.$$
The two claims of the lemma follow immediately from this computation.
\end{proof}

\subsubsection{Hyperbolic three-space} \label{sect:3.1.3}
We will be using the following model of hyperbolic three-space $\H^3$. Let $\H=\R\oplus\R i\oplus \R j\oplus \R k$ be the algebra of Hamiltonian quaternions, with the usual embedding of $\C=\R\oplus\R i$ into $\H$. Let
$$\H^3=\{(x_1,x_2,x_3)=x_1+x_2i+x_3j+0k\in\H:x_3>0\}$$
together with its standard hyperbolic metric $ds^2=(dx_1^2+dx_2^2+dx_3^2)/x_3^2$.
The action of $\SL_2(\C)$ on $\H^3$ is given by
$$\begin{bmatrix}a & b\\ c &d\end{bmatrix}\cdot q=(aq+b)(cq+d)^{-1}\quad\text{for all }\begin{bmatrix}a & b\\ c &d\end{bmatrix}\in\SL_2(\C)\text{ and }q\in\H.$$
This action is transitive, and the stabilizer of $j\in\H^3$ is the special unitary group $\SU(2)$, giving us the identification $\H^3\simeq\SL_2(\C)/\SU(2)$. For an element $g\in\SL_2(\C)$, its \emph{translation length} is defined to be
$$\length(g)=\inf\{d(x,g\cdot x):x\in\H^3\}$$
where $d(\cdot,\cdot)$ denotes the hyperbolic distance on $\H^3$. Note that translation length descends to a function of the conjugacy classes in $\SL_2(\C)$. If an element $g\in\SL_2(\C)$ is elliptic or parabolic, i.e.~$\tr(g)\in[-2,2]$, then $\length(g)=0$. Otherwise, up to conjugacy we may assume that
$$g=\begin{bmatrix}\lambda &0\\0 &\lambda^{-1}\end{bmatrix}$$
for some unique $\lambda\in\C$ with $|\lambda|>1$, and writing $\lambda=\exp((\ell+i\theta)/2)$ with $\ell,\theta\in\R$, we have $\length(g)=\ell$.

\subsubsection{Bruhat-Tits trees} \label{sect:3.1.4}
Let $F$ be a field with a discrete valuation $v:F^\times \to\Z$. We shall denote the Bruhat-Tits tree associated to $\SL_2(F)$ by $\Lambda$ (see \cite{serre} for background on Bruhat-Tits trees). We will work with the realization of $\Lambda$ as a Riemannian simplicial complex, where each edge has the Euclidean metric of the unit interval $[0,1]$. We note that, unless the residue field of the valuation is finite, the tree $\Lambda$ is not locally finite. The group $\SL_2(F)$ acts naturally on $\Lambda$ as a group of isometries. Given an element $g\in\SL_2(F)$, we define its \emph{translation length} to be
$$\length(g)=\inf\{d(x,g\cdot x):x\in\Lambda\}\in\Z_{\geq0}.$$
Note that $\length(g)$ depends only on the conjugacy class of $g$, and that $\length(g)=0$ if and only if $\tr(g)\in\Ocal$.

\subsubsection{Systoles of Riemannian surfaces} \label{sect:3.1.5}
Given a surface $S$ of type $(g,n)$ satisfying $3g+n-3>0$ with an arbitrary Riemannian metric, we define its \emph{systole} $\sys(S)$ to be the infimal length of an essential curve on $S$. We begin by recording the following results in systolic geometry due to Gromov \cite[Proposition 5.1.B]{gromov} and Balacheff-Parlier-Sabourau \cite[Theorem 6.10]{bps}.

\begin{theorem}[Gromov]
\label{gromov}
For any closed surface $S$ of genus $g\geq1$ with an arbitrary Riemannian metric, we have $\sys(S)^2\leq 2\Vol(S)$.
\end{theorem}

\begin{theorem}[Balacheff-Parlier-Sabourau]
\label{bps}
There is a constant $C_g$ such that, for any complete Riemannian surface $S$ of genus $g$ with $n$ ends with volume normalized to $2\pi|\chi(S)|$, there is a pants decomposition on $S$ with the sum of the lengths of the curves at most
$$C_gn\log(n+1).$$
\end{theorem}

Using these results, we deduce the following.

\begin{corollary}
\label{grocor}
There is a constant $C_{g,n}$ such that, for any Riemannian surface $S$ of type $(g,n)$ with $3g+n-3>0$, we have
$$\sys(S)^2\leq C_{g,n}(\Vol(S)+\ell_1^2+\cdots+\ell_n^2)$$
where $\ell_1,\cdots,\ell_n$ are the lengths of the boundary curves $c_1,\dots,c_n$ of $S$, respectively.
\end{corollary}

\begin{proof}
Assume first that $g\geq1$. For each $i=1,\dots,n$, let us attach a Riemannian disk $D_i$ to the boundary curve $c_i$ on $S$, satisfying the following two conditions:
\begin{enumerate}
	\item[\textup{(1)}] between any two points on the boundary of $D_i$, the shortest path between them lies on the boundary; and
	\item[\textup{(2)}] the area of $D_i$ is bounded by a fixed positive constant times the square of the perimeter.
\end{enumerate}
For instance, the disk $D_i$ can be taken to be the union of the cylinder $c_i\times[0,2\ell_i]$ with a spherical cap on one end. After attaching the disks $D_1,\dots,D_n$ to $S$ (and after a small perturbation to obtain a smooth Riemannian metric), the resulting closed Riemannian surface $(M,\sigma_M)$ has genus $g\geq1$ with volume
$$\Vol(M)\leq\Vol(S)+c\sum_i^n\ell_i^2$$
for some positive constant $c$. By Gromov's theorem (Theorem \ref{gromov}), there exists a shortest non-contractible curve $a$ on $M$ of length at most $\sqrt{2}\Vol(M)^{1/2}$. Note that $a$ lies entirely within $S$ by our assumption on the disks $D_i$, and in fact corresponds to an essential curve on $S$. Since the restriction of the metric $M$ to $S$ is the original metric (at least up to small perturbation near the boundary), we have
$$\length_S(a)^2\leq C(\Vol(S)+\ell_1^2+\dots+\ell_n^2)$$
for some suitable constant $C>0$ which is the desired result. The case $g=0$ can be similarly deduced, by attaching suitable cylinders with cusps to $S$ instead of disks, and using Theorem \ref{bps} in place of Theorem \ref{gromov}.
\end{proof}

\subsection{Harmonic maps}  \label{sect:3.2}
Let $f:M\to N$ be a smooth map of Riemannian manifolds with possible boundary. Its second fundamental form $\II^f$ is the covariant derivative of $df\in C^\infty(M,T^*M\otimes f^*TN)$:
$$\II^f=\nabla(df)\in C^\infty(M,T^*M\otimes T^*M\otimes f^*TN).$$
Here, $\nabla$ is the connection on $T^*M\otimes f^*TN$ induced from the Levi-Civita connection $\nabla^M$ on $M$ and the pullback $\nabla^f$ of the Levi-Civita connection on $N$. Explicitly,
$$\II^f(X,Y)=\nabla_X^f(df(Y))-df(\nabla_X^{M}Y)$$
for vector fields $X$ and $Y$ on $M$. If $f:M\to N$ is an isometric embedding of Riemannian manifolds, then in fact $\II^f(X,Y)$ defines a section of the normal bundle of $M$ in $N$. In such a case, we shall write $\II^{M}=\II^f$ if there is no risk of confusion.

Given a smooth map $f:M\to N$ of Riemannian manifolds, let us define the \emph{tension field} of $f$ to be the trace of the second fundamental form
$$\Delta f=\tr_M\II^f\in C^\infty(M,f^*TN).$$
Here, $\tr_M$ indicates the contraction of a section of $T^*M\otimes T^*M$ with the dual metric on $T^*M$, viewed as a section of $TM\otimes TM$. In other words, given a local orthonormal frame $\{e_1,\dots,e_n\}$ of $TM$,
$$\Delta f=\II^f(e_1,e_1)+\dots+\II^f(e_n,e_n).$$
The map $f$ is said to be \emph{harmonic} if its tension field vanishes identically, i.e.
$$\Delta f=0.$$
We mention in passing that the equation $\Delta f=0$ admits a variational interpretation as the Euler-Lagrange equation for the energy functional (see \cite{es} for details).

We shall consider a more general notion of ``twisted'' harmonic maps, or harmonic sections. Given a representation $\rho:\pi_1 M\to\Isom(N)$ of the fundamental group of $M$ into the group of isometries of $N$, we consider the fibre bundle
$$N_\rho=\bar M\times_\rho N$$
over $M$, where $\bar M$ is the universal cover of $M$. By local triviality of $N_\rho$, we may equip $N_\rho$ with the metric induced from those of $M$ and $N$, and extend the notions of second fundamental forms and harmonicity to sections $s:M\to N_\rho$. In particular, a section $s$ is \emph{harmonic} if
$$\Delta s=\tr_M(\nabla(Ds))=0$$
where $Ds$ is the vertical differential of $s$, and $\nabla$ is the connection on $T^*M\otimes s^* VN_\rho$, for $VN_\rho$ the vertical tangent bundle of $N_\rho$.

\subsubsection{Constructing harmonic sections} \label{sect:3.2.1}

Let $\Sigma$ be a surface of type $(g,n)$ with $3g+n-3>0$. Let $\sigma$ be a Riemannian metric on $\Sigma$ such that $S=(\Sigma,\sigma)$ has constant sectional curvature $-1$ and each of its boundary curves is geodesic of length $1$. Fix a representation $\rho:\pi_1 S\to\SL_2(\C)$ with Zariski dense image in $\SL_2(\C)$. The group $\SL_2(\C)$ acts on $\H^3$ via isometries. We form the associated fibre bundle
$$\H_\rho^3=\bar S\times_\rho\H^3$$
where $\bar S$ is the universal cover of $S$. We shall construct harmonic sections of $\H_\rho^3$ satisfying certain conditions along the boundary curves of $S$. We recall the solution of Dirichlet problem for harmonic sections.

\begin{lemma}
\label{harm}
Given any smooth section $s_0:S\to\H_\rho^3$, there is a smooth section $s_{\infty}$ in the homotopy class of $s_0$ satisfying
\begin{enumerate}
	\item[\textup{(1)}] $\Delta s_{\infty}=0$ and
	\item[\textup{(2)}] $s_{\infty}|_{\del S}=s_0|_{\del S}$.
\end{enumerate}
\end{lemma}

\begin{proof}
When the boundary of $S$ is empty, this was proved by Donaldson \cite{donaldson} using an adaptation of the heat flow argument of Eells-Sampson \cite{es} for harmonic maps to Riemannian manifolds of nonpositive curvature. The case where $S$ has nonempty boundary follows from work of Corlette \cite[Proof of Theorem 2.1]{corlette}, adapting the work of Hamilton \cite{hamilton} for harmonic maps on manifolds with boundary.
\end{proof}

We describe the boundary conditions we shall impose on our harmonic sections. Let $c_1,\dots,c_n$ be the geodesic boundary curves of $S$. For each $i=1,\dots,n$, let $\rho(c_i)$ denote the conjugacy class in $\SL_2(\C)$ determined by the monodromy of $\rho$ along $c_i$. Let us fix a geodesic parametrization $\gamma_i:[0,1]\to c_i$ of each boundary curve, and a trivialization
$$\gamma_i^*\H_\rho^3=\H^3\times[0,1]$$
such that the parallel transport
$\H^3=(\gamma_i^*\H_\rho^3)_0\to (\gamma_i^*\H_\rho^3)_1=\H^3$
is given by the action of the unique element $a_i\in\rho(c_i)$ lying in
$$\left\{\pm\mathbf1,\pm\begin{bmatrix}1 & 1\\ 0 & 1\end{bmatrix}\right\}\cup\left\{\begin{bmatrix}e^{i\theta}&0\\ 0 & e^{-i\theta}\end{bmatrix}:\theta\in(0,\pi)\right\}\cup\left\{\begin{bmatrix}\lambda&0\\ 0 & \lambda^{-1}\end{bmatrix}:\lambda\in\C,|\lambda|>1\right\}.$$
Let $h_i:c_i\to\H_\rho^3|_{c_i}$ be the unique section such that $h_i\circ\gamma_i:[0,1]\to\H^3\times[0,1]$ is the graph of the path $f_i:[0,1]\to\H^3$ defined as follows.
\begin{enumerate}
	\item[\textup{(1)}] If $a_i=\pm\mathbf1$, then $f_i\equiv j\in\H^3$.
	\item[\textup{(2)}] If $a_i$ is elliptic with unique fixed point $q\in\H^3$, then $f_i\equiv q$.
	\item[\textup{(3)}] If $a_i$ is parabolic then
	$$f_i(t)=\begin{bmatrix}1 & t\\0 & 1\end{bmatrix}j\text{ for all }t\in[0,1].$$
	The curve underlying $f_i$ has length $1$.
	\item[\textup{(4)}] If $a_i=\left[\begin{smallmatrix}\lambda&0\\0&\lambda^{-1}\end{smallmatrix}\right]$ (with $\lambda\in\C$ and $|\lambda|>1$) is hyperbolic, then
	$$f_i(t)=\begin{bmatrix}\lambda^t & 0\\0 & \lambda^{-t}\end{bmatrix}j\text{ for all }t\in[0,1]$$
	This does not depend on the choice of branch for the logarithm of $\lambda$. Note that $f_i$ is a geodesic path of length equal to $\length(a_i)$.
\end{enumerate}
Let us consider the section
\begin{equation}
\label{sectiondef}
h=h_1\sqcup\dots\sqcup h_n:\del S=c_1\sqcup\dots\sqcup c_n\to\H_\rho^3|_{\del S}.
\end{equation}
with the $h_i$ defined as above.

\begin{lemma}
\label{bound}
There is a smooth section $s:S\to\H_\rho^3$ satisfying $s|_{\del S}=h$.
\end{lemma}

\begin{proof}
We shall keep the notations introduced above. For each $i=1,\dots,n$, let $N_i$ be a closed half-collar neighborhood of the boundary curve $c_i$ such that $N_i\cap N_j=\emptyset$ for $i\neq j$. Let $n_i:[0,1]\times[0,1]\to N_i$ be a parametrization of $N_i$ with $n_i(s,1)=\gamma_i(s)$ and $n_i(0,t)=n_i(1,t)$ for all $s,t\in[0,1]$. We have a trivialization
$$n_i^*\H_\rho^3=\H^3\times[0,1]\times[0,1]$$
that restricts to the trivialization $\gamma_i^*\H_\rho^3=\H^3\times[0,1]$ chosen above. Thus, we have a bijective correspondence between the set of continuous sections $N_i\to\H_\rho^3|_{N_i}$ and the set $S_i$ of continuous maps $g:[0,1]^2\to\H^3$
satisfying $g(1,t)=a_i\cdot g(0,t)$ for every $t\in[0,1]$. Given any $g_0\in S_i$, there is a homotopy $g:[0,1]^2\times[0,1]\to\H^3$ such that $g_u=g(-,-,u)\in S_i$ for any $u\in[0,1]$ and
\begin{enumerate}
	\item[\textup{(1)}] $g(s,t,0)=g_0(s,t)$,
	\item[\textup{(2)}] $g(s,0,u)=g_0(s,0)$, and
	\item[\textup{(3)}] $g(s,1,1)=f_i(s)$
\end{enumerate}
for every $s,t,u\in[0,1]$. This can be seen, for instance, from the fact that between any two paths $m_0,m_1:[0,1]\to\H^3$ satisfying $m_0(1)=a_i\cdot m_1(0)$ for $j=0,1$ there is a homotopy $m:[0,1]\times[0,1]\to\H^3$ satisfying $m(-,j)=m_j$ for $j=0,1$ and $m(1,t)=a_i\cdot m(0,t)$ for every $t\in[0,1]$. Therefore, given any section $s_0:S\to\H_\rho^3$, considering the restrictions $s_0|_{N_i}$ for each $i=1,\dots,n$ and applying the above homotopy, we obtain a continuous section $s:S\to\H_\rho^3$ satisfying $s|_{\del S}=h$. It is easily seen that we may arrange $s$ to be smooth.
\end{proof}

\begin{corollary}
\label{harmco}
There is a harmonic section $s:S\to\H_\rho^3$ satisfying $s|_{\del S}=h$.
\end{corollary}

\begin{proof}
This follows by combining Lemmas \ref{harm} and \ref{bound}.
\end{proof}

\subsection{Systoles of local systems}  \label{sect:3.3}
\subsubsection{Archimedean systoles} \label{sect:3.3.1}
Let $\Sigma$ be a surface of type $(g,n)$ with $3g+n-3>0$, having boundary curves $c_1,\dots,c_n$. Given a representation $\rho:\pi_1\Sigma\to\SL_2(\C)$ or a corresponding local system on $\Sigma$, we define its \emph{systole} to be
$$\sys(\rho)=\inf\{\length\rho(a):\text{$a\subset\Sigma$ is an essential curve}\}$$
where $\length\rho(a)$ denotes the translation length of the conjugacy class $\rho(a)$ in $\SL_2(\C)$ acting on the hyperbolic three-space $\H^3$, as defined in Section \ref{sect:3.1.3}. This recovers the usual notion of systole for hyperbolic surfaces (see Section \ref{sect:3.1.5}) when $\rho:\pi_1\Sigma\to\SL_2(\R)$ is a Fuchsian representation arising from a hyperbolic structure on $\Sigma$ with geodesic boundary. By the relation between translation lengths and traces of elements in $\SL_2(\C)$ given in Section \ref{sect:3.1.3}, we see that Theorem \ref{theorem2} follows from the following claim (see Section \ref{sect:3.3.3}).

\begin{theorem}
\label{sysarch}
There exists a real number $C>0$ such that, for any representation $\rho:\pi_1\Sigma\to\SL_2(\C)$, we have $\sys(\rho)^2\leq C(2\pi|\chi(\Sigma)|+\sum_{i=1}^n\length\rho(c_i)^2)$.
\end{theorem}

\begin{proof}
Let us first choose a Riemannian metric $\sigma$ on $\Sigma$ such that $S=(\Sigma,\sigma)$ is a hyperbolic surface with geodesic boundary curves $c_1,\dots,c_n$, each of length $1$. Let us fix a representation
$\rho:\pi_1\Sigma\to\SL_2(\C)$, which we assume has Zariski dense image in $\SL_2(\C)$ since otherwise the theorem is easy to deduce; this is easily seen from the classification of algebraic subgroups of $\SL_2(\C)$. Consider the fibre bundle
$$\H_\rho^3=\bar\Sigma\times_\rho\H^3$$
where $\bar \Sigma$ is the universal cover of $\Sigma$. We may equip $\H_\rho^3$ with the Riemannian metric induced from the hyperbolic metric on $\H^3$ and the metric $\sigma$ on $S$. Let us fix a section $h:\del S\to\H_\rho^3$ as given by equation (\ref{sectiondef}) in Section \ref{sect:3.2}. By Corollary \ref{harmco}, there exists a harmonic section
$$s:S\to\H_\rho^3$$
satisfying $s|_{\del S}=h$. For each $\epsilon>0$, let $S(\epsilon)$ be the surface equipped with the scaled metric $\epsilon\sigma$, and let $\H_\rho^3[\epsilon]$ be the fibre bundle with metric induced from the hyperbolic metric $\sigma_{\H^3}$ on $\H^3$ and the metric on $S(\epsilon)$. Note that $s:S(\epsilon)\to\H_\rho^3[\epsilon]$ remains harmonic. Let
$$f:\bar\Sigma\to\H^3$$
be the $\rho$-equivariant map corresponding to $s$. The nonnegative symmetric tensor $f^*\sigma_{\H^3}$ on $\bar\Sigma$ descends to $\Sigma$ since $f$ is $\rho$-equivariant. Let us consider the metric
$$\sigma(\epsilon)=f^*\sigma_{\H^3}+\epsilon\sigma$$
on the surface $\Sigma$. In fact, $\sigma(\epsilon)$ is precisely metric induced from $\H_\rho^3[\epsilon]$ on the image $M(\epsilon)=s(S(\epsilon))$ of the harmonic section $s:S(\epsilon)\to\H_\rho^3[\epsilon]$. Note that
$$\length\rho(a)\leq\length_{M(\epsilon)}(a)$$
for any essential closed geodesic $a$ on $M(\epsilon)$, the left hand side being the translation length of $\rho(a)$ as it acts on $\H^3$. In particular, we see that $\sys(\rho)\leq \sys(M(\epsilon))$ for every $\epsilon>0$. Now, note that the lengths of the boundary curves $c_1,\dots,c_n$ of $M(\epsilon)$ satisfy
$$\lim_{\epsilon\to0}\length_{M(\epsilon)}(c_i)=\left\{\begin{array}{l l}
1 & \text{if $\rho(c_i)$ is parabolic, and}\\
\length(\rho(c_i))&\text{otherwise.}\end{array}\right.$$
Since $\sigma(\epsilon)$ tends to the symmetric tensor $f^*\sigma_{\H^3}$ uniformly as $\epsilon\to0$, we have
$$\lim_{\epsilon\to0}\Vol(M(\epsilon))=\int_\Sigma d\vol_{f^*\sigma_{\H^3}}.$$
(The ``volume form'' $d\vol_{f^*\sigma_{\H^3}}$ associated to the nonnegative symmetric tensor $f^*\sigma_{\H^3}$ may vanish at some places.) 
Therefore, by Corollary \ref{grocor}, for all sufficiently small $\epsilon>0$, we have
$$\sys(M(\epsilon))^2\leq C_1\left(\int_\Sigma d\vol_{f^*\sigma_{\H^3}}+\sum_{i=1}^n\max\{1,\length(\rho(c_i))^2\}\right)$$
for some absolute constant $C_1>0$. Therefore, it will suffice to show that
$$\int_\Sigma d\vol_{f^*\sigma_{\H^3}}\leq 2\pi|\chi(\Sigma)|+c$$
for some absolute constant $c>0$. The remainder of the proof is devoted to demonstrating this claim (due essentially to Reznikov), combining computations in the proof of Lemma 2.5 of Deroin-Tholozan \cite{dt} with an argument of Reznikov \cite{reznikov}.

By a result of Sampson \cite[Corollary of Theorem 3]{sampson}, the open subset $U\subseteq\Sigma$ corresponding to the locus in $\bar\Sigma$ where $df$ has full rank is either empty or dense. Since the former case is trivial, we shall assume that $df$ has full rank on a dense open subset of $\Sigma$. Locally, given a choice of orthonormal frame $e_1,e_2$ of $TM(\epsilon)$, by Gauss's theorem \cite[Chapter 6, Theorem 2.5]{docarmo} we have
$$K^{M(\epsilon)}=K^{\H_\rho^3[\epsilon]}(TM(\epsilon))+\langle\II^{M(\epsilon)}(e_1,e_1),\II^{M(\epsilon)}(e_2,e_2)\rangle-\|\II^{M(\epsilon)}(e_1,e_2)\|^2$$
where $\II^{M(\epsilon)}$ is the second fundamental form of $M(\epsilon)$ immersed in $\H_\rho^3[\epsilon]$, and the inner product is with respect to the metric on the normal bundle $NM(\epsilon)$ of $M(\epsilon)$. Choosing a local orthonormal frame $\{n_1,n_2,n_3\}$ of $NM(\epsilon)$, we have 
\begin{align*}
&\langle\II^{M(\epsilon)}(e_1,e_1),\II^{M(\epsilon)}(e_2,e_2)\rangle-\|\II^{M(\epsilon)}(e_1,e_2)\|^2\\
&\quad=\sum_{k=1}^3\det(\langle\II^{M(\epsilon)}(e_i,e_j),n_k\rangle)\\
&\quad=\frac{1}{\Jac(s)^2}\sum_{k=1}^3\det_{\epsilon\sigma}(\langle\II^{M(\epsilon)}(ds(-),ds(-)),n_k\rangle)\\
&\quad=\frac{3}{\Jac(s)^2}\E\left(\det_{\epsilon\sigma}\langle\II^{M(\epsilon)}\circ ds,n\rangle\right)
\end{align*}
where the average $\E$ is taken over all unitary normal vectors $n$ in $NM(\epsilon)$. Note that we have by definition of the second fundamental form of $s$
\begin{align*}
\II^s(X,Y)&=\nabla_{X}^{s}Ds(Y)-Ds(\nabla_X^{\epsilon\sigma}Y)=\nabla_X^{s^*\sigma(\epsilon)}ds(Y)-ds(\nabla_X^{\epsilon\sigma}Y)\\
&=\II^{M(\epsilon)}(ds(X),ds(Y))+ds\left(\nabla_X^{s^*\sigma(\epsilon)}Y-\nabla_{X}^{\epsilon\sigma}Y\right).
\end{align*}
But note that
$$\II^{M(\epsilon)}(ds(X),ds(Y))\quad\text{and}\quad ds\left(\nabla_X^{s^*\sigma(\epsilon)}Y-\nabla_{X}^{\epsilon\sigma}Y\right)$$
are orthogonal. Therefore, taking the trace of both sides and using the harmonicity of $s$ (i.e.~$\tr_\sigma\II^s=0$), we have $\tr_{\epsilon\sigma}(\II^{M(\epsilon)}\circ ds)=0$, and \emph{a fortiori}
$$\tr_{\epsilon\sigma}\langle\II^{M(\epsilon)}(ds(-),ds(-)), n\rangle=0$$
for every normal vector $n\in NM(\epsilon)$. Since $\langle\II^{M(\epsilon)}(ds(-),ds(-)),n\rangle$ is a symmetric tensor, its eigenvalues (with respect to the metric $\epsilon\sigma$ on the surface) are real, and the vanishing of trace therefore implies $\det_{\epsilon\sigma}\langle\II^{M(\epsilon)}(ds(-),ds(-)),n\rangle\leq0$. Thus, we conclude from Gauss's theorem that
$$-K^{\H_\rho^3[\epsilon]}(TM(\epsilon))\leq -K^{M(\epsilon)}.$$
Applying the Gauss-Bonnet formula, we therefore have
\begin{align*}
\int_{M(\epsilon)}-K^{\H_\rho^3[\epsilon]}(TM(\epsilon))\:d\vol_{\sigma(\epsilon)}&\leq\int_{M(\epsilon)}(-K^{M(\epsilon)})\:d\vol_{\sigma(\epsilon)}\\
&\leq2\pi|\chi(\Sigma)|+\int_{\del M(\epsilon)} k^{M(\epsilon)}\:ds_{\sigma(\epsilon)}
\end{align*}
where $k^{M(\epsilon)}$ is the signed curvature and $ds_{\sigma(\epsilon)}$ is the line element of $\del M(\epsilon)$. Note that $-K^{\H_\rho^3[\epsilon]}$ is everywhere nonnegative by part (1) of Lemma \ref{scurv}, since $\H^3$ and $S(\epsilon)$ have nonpositive sectional curvature. Furthermore, we have
$$\lim_{\epsilon\to0}K^{\H_\rho^3[\epsilon]}(TM(\epsilon))=K^{\H^3}(df(T\Sigma))=-1$$
on the dense open set $U\subset\Sigma$ where $df$ has full rank, by part (2) of Lemma \ref{scurv}. Recalling also that $\lim_{\epsilon\to0}\sigma(\epsilon)=f^*\sigma_{\H^3}$, we apply Fatou's lemma to deduce that
\begin{align*}
\int_\Sigma d\vol_{f^*\sigma_{\H^3}}&\leq\liminf_{\epsilon\to0}\int_{M(\epsilon)}-K^{\H_\rho^3[\epsilon]}(TM(\epsilon))\:d\vol_{\sigma(\epsilon)}\\
&\leq2\pi|\chi(\Sigma)|+\liminf_{\epsilon\to0}\int_{\del M(\epsilon)}k^{M(\epsilon)}ds_{\sigma(\epsilon)}.
\end{align*}
It therefore remains only to bound the total curvature of $\del M(\epsilon)$ for all sufficiently small $\epsilon$. For $\epsilon>0$, let $c_i(\epsilon)$ be the curve $c_i$ equipped with the metric such that its length is $\epsilon$. We claim that
$$\lim_{\epsilon\to0}\int_{c_i(\epsilon)}k^{\H_\rho^3[\epsilon]}\:ds_{\sigma(\epsilon)}=\left\{\begin{array}{l l}0 & \text{if $\rho(c_i)$ is central, elliptic, or hyperbolic,}\\
1 &\text{if $\rho(c_i)$ is parabolic.}\end{array}\right.$$
Indeed, we note that $h:c_i(\epsilon)\to\H_\rho^3[\epsilon]|_{c_i}$ is a geodesic path if $\rho(c_i)$ is central, elliptic, or hyperbolic. We shall therefore assume that $\rho(c_i)$ is parabolic. In terms of the global coordinates $(x_1,x_2,x_3)$ on $\H^3$ given in Section \ref{sect:3.1.3}, the Christoffel symbols for the Levi-Civita connection on $\H^3$ are
\begin{align*}
\Gamma_{ij,k}&=\langle\nabla_{\del/\del x_i}(\del/\del x_j),\del/\del x_k\rangle\\
&=\frac{1}{2}\left(\frac{\del}{\del x_j}\left\langle\frac{\del}{\del x_i},\frac{\del}{\del x_k}\right\rangle+\frac{\del}{\del x_i}\left\langle\frac{\del}{\del x_j},\frac{\del}{\del x_k}\right\rangle-\frac{\del}{\del x_k}\left\langle\frac{\del}{\del x_i},\frac{\del}{\del x_j}\right\rangle\right).
\end{align*}
Thus, we have $\Gamma_{11,1}=\Gamma_{11,2}=0$ and $\Gamma_{11,3}=1/x_3$, and
$$\frac{D}{dt}\frac{d f_i}{dt}=\Gamma_{11,3}x_3^2\frac{\del}{\del x_3}=\frac{1}{x_3}\frac{\del}{\del x_3}.$$
In particular, we have $k^{\H^3}\equiv 1$ on the image of $f_i:[0,1]\to\H^3$ (in the notation of Section \ref{sect:3.2.1}). It is easy to observe that $k^{\H_\rho^3[\epsilon]}$ tends to $k^{\H^3}$ uniformly as $\epsilon\to0$. The desired result follows from this.
\end{proof}

\subsubsection{Nonarchimedean systoles} \label{sect:3.3.2}
Let $\Sigma$ be a surface of type $(g,n)$ with $3g+n-3>0$, having boundary curves $c_1,\dots,c_n$. Let $\Ocal$ be a discrete valuation ring, and let $F$ be its fraction field. Given a representation $\rho:\pi_1\Sigma\to\SL_2(F)$ or a corresponding local system on $\Sigma$, we define its \emph{systole} to be
$$\sys(\rho)=\inf\{\length\rho(a):\text{essential curve $a\subset\Sigma$}\}$$
where $\length\rho(a)$ denotes the translation length of the conjugacy class $\rho(a)$ acting on the Bruhat-Tits tree $\Lambda$ of $\SL_2(F)$, as defined in Section \ref{sect:3.1.4}. Theorem \ref{theorem3} is a special case of the following result (see Section \ref{sect:3.3.3}).

\begin{theorem}
\label{sysprop}
There exists a real number $C>0$ such that, for any representation $\rho:\pi_1\Sigma\to\SL_2(F)$, we have
$\sys(\rho)^2\leq C\cdot \sum_{i=1}^n(\length\rho(c_i))^2$.
\end{theorem}

\begin{proof}
Let us consider the fiber bundle
$$\Lambda_\rho=\bar\Sigma\times_\rho\Lambda$$
over $\Sigma$, where $\bar\Sigma$ is the universal cover of $\Sigma$. Our first step is to construct a section $s:\Sigma\to\Lambda_\rho$ satisfying the following two conditions.
\begin{enumerate}
	\item[\textup{(1)}] The subset $B\subset\Sigma$ where $s$ fails to be smooth is finite.
	\item[\textup{(2)}] For each $i=1,\dots,n$, given some parametrization $\gamma_i:[0,1]\to c_i$ of the boundary curve $c_i$ and trivialization
$$\gamma_i^*\Lambda_\rho\simeq \Lambda\times [0,1]$$
the restriction $\gamma_i^*s:[0,1]\to \Lambda\times [0,1]$ is the graph of a parametrization of a geodesic curve whose length equals $\length(\rho(c_i))$.
\end{enumerate}
In condition (1) above and our arguments, in consideration of the smoothness of $s$ we shall take into account the various ``apartments'' of $\Lambda$, as in the work of Gromov and Schoen (see \cite[p.178]{grs} for an elucidation). The construction of $s$ can be given as follows. Let $x_0\in\Sigma$ be a fixed base point lying on the interior of the surface. For each $j=1,\dots,n$, let $x_j$ be a base point on $c_j$, and let $\gamma_j$ be a simple loop on $c_j$ based at $x_j$, so that the interior of $\Sigma$ lies to the left as one traves along $\gamma_j$. Let $\alpha_1,\dots,\alpha_{2g}$ be simple loops based on $x_0$, and let $\beta_j$ for each $j=1,\dots,n$ be a simple path from $x_0$ to $x_j$ not intersecting any of the aforementioned loops except (at the endpoints), such that the sequence
$$(\alpha_1,\alpha_2,\dots,\alpha_{2g-1},\alpha_{2g},\beta_1\gamma_1\beta_1^{-1},\dots,\beta_n\gamma_{n}\beta_n^{-1})$$
of loops gives the standard generators of $\pi_1(\Sigma,x_0)$ pairwise intersecting only at the base point $x_0$. Fix a trivialization $\Psi_0:(\Lambda_\rho)_{x_0}\simeq\Lambda$ of the fiber of $\Lambda_\rho$ above $x_0$. By parallel transport this induces trivializations
\begin{itemize}
	\item $\alpha_i^*\Lambda_\rho\simeq[0,1]\times\Lambda$ for $i=1,\dots,2g$, and
	\item $\beta_j^*\Lambda_\rho\simeq[0,1]\times\Lambda$ for $j=1,\dots,n$,
\end{itemize}
agreeing with $\Psi_0$ at $0$. For $j=1,\dots,n$, let $\Psi_j:(\Lambda_\rho)_{x_j}\simeq\Lambda$ be obtained from the second set of trivializations at $1$. By parallel transport, we have a trivialization
\begin{itemize}
	\item $\gamma_j^*\Lambda_\rho\simeq[0,1]\times\Lambda$ for $j=1,\dots,n$,
\end{itemize}
agreeing with $\Psi_j$ at $0$. Let $a_i$ be the curve underlying $\alpha_i$. Similarly, let $b_j$ denote the $1$-submanifold of $\Sigma$ underlying $\beta_i$. By construction, the curve underlying the loop $\gamma_j$ is the boundary curve $c_j$. Let
$$C=\bigcup_{i=1}^{2g} a_i\cup\bigcup_{j=1}^n(b_j\cup c_j).$$
Fix $v_0\in\Lambda$. We have a section $t:C\to {\Lambda_\rho}|_C$ given by the following properties under the various trivializations above:
\begin{itemize}
	\item $t(x_0)=v_0$ under the trivialization $\Psi_0$,
	\item the composition $t\circ\alpha_i$ corresponds to the graph of the geodesic path on $\Lambda$ from $v_0$ to $\rho(\alpha_j)\cdot v_0$,
	\item the composition $t\circ\gamma_j$ corresponds to the graph of a geodesic path on $\Lambda$ (from point $v_j$ to $v_j'$, say) whose length equals $\length\rho(c_j)$, and
	\item the composition $t\circ\beta_j$ corresponds to the graph of the geodesic path on $\Lambda$ from $v_0$ to the point $v_j$ chosen above.
\end{itemize}
Our goal is to extend $t$ to a section on the entire surface $\Sigma$ as follows. Note that the latter corresponds to a $\rho$-equivariant map $\bar\Sigma\to\Lambda$. The complement of the preimage $\bar C$ of $C$ in $\bar\Sigma$ divides $\bar\Sigma$ into a countable collection of components, each of which is an open fundamental domain for the covering $\bar\Sigma\to\Sigma$. Let $\Fcal$ be one of them. Since $t$ already specifies for us a $\rho$-equivariant map $\bar C\to\Lambda$ (and in particular a map $\del\Fcal\to\Lambda$ we shall call $f$), to construct a section $s:\Sigma\to\Lambda_\rho$ extending $t$ it remains only to construct a map $F:\bar\Fcal\to\Lambda$ such that $F|_{\del\Fcal}=f$. Note that the image of $f$ is contained in a subtree $T\subset\Lambda$ with finitely many edges.

Let $D$ be the closed unit disk. Let us fix a homeomorphism $\Fcal\simeq D$ which is a diffeomorphism on the interior and restricts to a diffeomorphism on each of the ``edges'' of $\del\Fcal$. Fix $0<\epsilon<1$. There is then a homotopy $g:[0,\epsilon]\times \del D\to T$ such that $g_0=f$ under the identification above, and $g|_{(0,\epsilon]}$ is smooth. (This serves to ``round out the corners'' of the map $f$.) Consideration of deformation retractions of $T$ to a point $p\in T$ (or invoking a picture as in \cite[p.178]{grs}) shows that we may extend $g$ to a homotopy $[0,1]\times \del D\to T$ such that:
\begin{enumerate}
	\item [(i)] $g|_{(0,1]}$ is smooth at all but finitely many points, and
	\item [(ii)] $g_1\equiv p$.
\end{enumerate}
By construction, this factors through the quotient
$$F:\bar\Fcal\simeq D\simeq([0,1]\times \del D)/((x,1)\sim(x',1))\to T$$
where the homeomorphism from $D\simeq ([0,1]\times \del D)/((x,1)\sim(x',1))$ appearing above is chosen in the obvious way (so that in particular it is a diffeomorphism away from the origin $o\in D$). This gives rise to a section $s:\Sigma\to\Lambda_\rho$ which is smooth away from $C$ and finitely many points. By considering restrictions of $\Lambda_\rho$ over open neighborhoods of $a_i$, $b_j$, and $c_j$, we see that $s$ may be deformed to a section which is smooth away from a finite set $B\subset\Sigma$. This completes the construction of the section $s:\Sigma\to\Lambda_\rho$ satisfying the conditions (1) and (2) mentioned in the beginning.

Let $\sigma$ be an arbitrary Riemannian metric on $\Sigma$. On the complement $\Sigma\setminus B$ of the finite set $B$ given as in (1), we have a Riemannian metric given by
$$\sigma'(\epsilon)=s^*(\sigma_\Lambda+\epsilon\sigma)$$
for $\epsilon>0$, where $\sigma_\Lambda$ is the Euclidean piecewise metric on the Bruhat-Tits tree $\Lambda$. Moreover, for any curve $a\subset\Sigma$ its length with respect to $\sigma'(\epsilon)$ is well-defined. Setting $s(\epsilon)=\inf\{\length(a,\sigma'(\epsilon)):\text{$a\subset\Sigma$ essential curve}\}$, we find that
\begin{align*}
\tag{a}\sys(\rho)\leq s(\epsilon).
\end{align*}
Our next step is to modify $\sigma'(\epsilon)$ to obtain a genuine Riemannian metric on $\Sigma$. More precisely, let $\{D_x\}_{x\in B}$ be a collection of pairwise disjoint closed disks lying on the interior of $\Sigma$, such that, for each $x\in B$,
\begin{enumerate}
	\item[(i)] $D_x$ contains $x$ in its interior,
	\item[(ii)] $\Vol(D_x,\sigma'(\epsilon))<\epsilon^2$, and
	\item[(iii)] $\length(\del D_x,\sigma'(\epsilon))<\epsilon$.
\end{enumerate} 
Let $\sigma(\epsilon)$ be a Riemannian metric on $\Sigma$ obtained by extending the metric $\sigma'(\epsilon)$ on the complement of $\bigcup_{x\in B}D_x$ to $\Sigma$ in such a way that $\sigma(\epsilon)|_{D_x}$ is the round metric on a half-sphere with base circumference equal to $\length(\del D_x,\sigma'(\epsilon))$, up to negligible smoothing. By our construction, we have
\begin{align*}
\tag{b}\Vol(\Sigma,\sigma(\epsilon))=\Vol(\Sigma\setminus B,\sigma'(\epsilon))+O(\epsilon^2)
\end{align*}
and
\begin{align*}
\tag{c}s(\epsilon)\leq\sys(\Sigma,\sigma(\epsilon)).
\end{align*}
To observe the latter, we note that a shortest path between any two points on the boundary of the half-sphere $(D_x,\sigma(\epsilon))$ may be taken to be on the boundary $\del D_x$. Using the fact that $\sigma'(\epsilon)$ tends to $s^*\sigma_\Lambda$ uniformly on $\Sigma'$ as $\epsilon$ tends to $0$ and the fact that $\Lambda$ is a one-dimensional Riemannian simplicial complex, we conclude that
$$\lim_{\epsilon\to0}\Vol(\Sigma\setminus B,\sigma'(\epsilon))=0.$$
Therefore, taking $\epsilon>0$ sufficiently small and combining the observations (a), (b), and (c) above with Corollary \ref{grocor}, we obtain the desired result.
\end{proof}

\begin{remark}
The essence of our argument in the proof of Theorem \ref{sysprop} is that, given a (twisted) map $f:M\to N$ from a compact manifold $M$ to another manifold $N$ of strictly smaller dimension, the pullback $f^*\sigma_N$ of any Riemannian metric $\sigma_N$ on $N$ will be a degenerate symmetric form on $M$ and therefore have zero volume element. Combined with systolic inequalities for \emph{essential} manifolds (in the sense of Gromov \cite{gromov}), this has an elementary consequence that, given any smooth map
$$f:M\to N$$
from a closed essential manifold $M$ to (for instance) another manifold $N$ of strictly smaller dimension without arbitrarily short nontrivial loops, the induced map on the fundamental groups $\pi_1 M\to\pi_1 N$ cannot be injective.
\end{remark}

\subsubsection{Proofs of Theorems \ref{theorem2} and \ref{theorem3}}\label{sect:3.3.3}
As before, let $\Sigma$ be a surface of type $(g,n)$ with $3g+n-3>0$, having boundary curves $c_1,\dots,c_n$. We now restate and deduce Theorems \ref{theorem2} and \ref{theorem3} as special cases of Theorems \ref{sysarch} and \ref{sysprop}, respectively.

\begin{thm}{1.2}
For each $B\geq0$, there is $A\geq0$ depending continuously on $B$ such that, given any representation $\rho:\pi_1\Sigma\to\SL_2(\C)$ whose boundary traces all have absolute value at most $B$, there is an essential curve $a\subset\Sigma$ with $|\tr\rho(a)|\leq A$.
\end{thm}

\begin{proof}
Let $B\geq0$ be given, and let $\rho:\pi_1\Sigma\to\SL_2(\C)$ be a representation such that $|\tr\rho(c_i)|\leq B$ for every boundary curve $c_i$ of $\Sigma$. We see, using the relationship (see Section \ref{sect:3.1.3}) between the trace of an element in $\SL_2(\C)$ and its translation length on $\H^3$, that there is a constant $B'$, depending continuously on $B$, with $\length\rho(c_i)\leq B'$ for each boundary curve $c_i$. Theorem \ref{sysarch} then shows that
$$\sys(\rho)^2\leq C(2\pi|\chi(\Sigma)|+nB'^2)$$
for some absolute constant $C$. By the relationship between trace and translation length, the above inequality shows that there is a constant $A$, depending continuously on $B$, such that we have $|\tr\rho(a)|\leq A$ for some essential curve $a\subset\Sigma$.
\end{proof}

\begin{thm}{1.3}
Let $\Ocal$ be a discrete valuation ring with fraction field $F$. Given any representation $\rho:\pi_1\Sigma\to\SL_2(F)$ whose boundary traces all take values in $\Ocal$, there is an essential curve $a\subset\Sigma$ with $\tr\rho(a)\in\Ocal$.
\end{thm}

\begin{proof}
Let $\rho:\pi_1\Sigma\to\SL_2(F)$ be as in the hypothesis of the theorem. Recall that, for any $g\in\SL_2(F)$, its translation length on the Bruhat-Tits tree is zero if and only if and only if $\tr(g)\in\Ocal$. Thus, by Theorem \ref{sysprop}, we have $\sys(\rho)^2\leq C\cdot n\cdot 0^2=0$ for some absolute constant $C$. Since the translation length of an element in $\SL_2(F)$ is a nonnegative integer, this means that there is an essential curve $a\subset\Sigma$ such that $\length\rho(a)=0$, or in other words $\tr\rho(a)\in\Ocal$, as desired.
\end{proof}

\section{Compactness criterion}  \label{sect:4}
\subsection{Overview}\label{sect:4.0}
Let $\Sigma$ be a surface of type $(g,n)$ satisfying $3g+n-3>0$. The purpose of this section is to prove Theorem \ref{theorem4}. The intuition behind our proof is the following geometric picture. By the work of Goldman (see for example \cite{goldman3}), the variety $X$ carries an algebraic Poisson structure for which the subvarieties $X_k$, $k\in\A^n(\C)$, form symplectic leaves. Given a pants decomposition $P$ of $\Sigma$, the immersion $P\to\Sigma$ induces a morphism
$$
\tr_P:X_k\to X(P)\simeq\A^{3g+n-3}
$$
which forms an integrable system in the sense of Hamiltonian mechanics. The generic fibre of this morphism is an algebraic torus of dimension $3g+n-3$, preserved by the pairwise commuting actions of Dehn twists along curves in $P$. (Moreover, these Dehn twists actions would arise from suitable Hamiltonian flows associated to the trace functions.) This setup was notably utilized by Goldman \cite{goldman} to prove the ergodicity of mapping class group dynamics on the locus of $\SU(2)$-local systems in $X_k(\R)$. In our context, this setup suggests that the task of descent for points $\rho\in X_k(\C)$ divides into two steps:
\begin{enumerate}
	\item[\textup{(1)}] (``action variables'') Find a transformation $\gamma\in\Gamma$ so that $\tr_P(\gamma^*\rho)\in\A^n(\C)$ lies in some predetermined compact set, and
	\item[\textup{(2)}] (``angle variables'') Descend along fibers of $\tr_P$ using the Dehn twist actions.
\end{enumerate}
Unraveling the definitions, it can be observed (see Corollary \ref{pantsco} below) that the boundedness result (Theorem \ref{theorem2}) for systoles of $\SL_2(\C)$-local systems allows us to achieve the first step without difficulty (for some pants decomposition $P$). In fact, combining this with an in-depth analysis of fibers of $\tr_P$ in \cite{whang4}, we can give a proof of Theorem \ref{theorem1} along the above lines, as shown below.

\begin{corollary}
\label{pantsco}
For each $B\geq0$, there is $A\geq0$ depending continuously on $B$ such that, given any representation $\rho:\pi_1\Sigma\to\SL_2(\C)$ whose boundary traces all have absolute value at most $B$, there is a pants decomposition $P=a_1\sqcup\dots\sqcup a_{3g+n-3}$ of $\Sigma$ with $|\tr\rho(a_i)|\leq A$ for all $i=1,\dots,3g+n-3$.
\end{corollary}

\begin{proof}
Let $\rho:\pi_1\Sigma\to\SL_2(\C)$ be given as above. By Theorem \ref{theorem2} there is a constant $C_1=C_1(B)$ and an essential curve $a_1\subset\Sigma$ with $|\tr\rho(a_1)|\leq C_1$. Considering the restriction of $\rho$ to (a component of) the surface $\Sigma|a_1$, again by Theorem \ref{theorem2} there is a constant $C_2=C_2(C_1,B)$ and an essential curve $a_2$ of $\Sigma|a_1$ with $|\tr\rho(a_2)|\leq C_2$. (We note that, since there are only finitely many mapping class group equivalence classes of essential curves on $\Sigma$, considered up to isotopy, we can take $C_2$ to be independent of the class of the curve $a_1\subset\Sigma$.) We next consider the restriction of $\rho$ to $(\Sigma|a_1)|a_2$, and so on. Repeating this enough times, we obtain a constant $A=\max\{C_1,\dots,C_{3g+n-3}\}$ depending only on $B$ and a pants decomposition $P=a_1\sqcup a_2\sqcup\dots\sqcup a_{3g+n-3}$ such that $|\tr\rho(a_i)|\leq A$ for all $i=1,\dots,3g+n-3$.
\end{proof}

\begin{thm}{1.1}
The nondegenerate integral points in $X_k(\Z)$ consist of finitely many mapping class group orbits. There exists a parabolic proper closed subvariety of $X_k$ whose orbit gives precisely the locus of degenerate points on $X_k$.
\end{thm}

\begin{proof}
We note that, up to mapping class group equivalence and isotopy, there are only finitely many pants decompositions of $\Sigma$. Let $\Pcal=\{P_1,\dots,P_m\}$ be a complete set of representatives. By Corollary \ref{pantsco}, there is a compact subset $M\subset\C^{3g+n-3}$ such that, for any $\rho\in X_k(\C)$, we have $\tr_P(\gamma^*\rho)\in M$ for some $\gamma\in\Gamma$ and $P\in\Pcal$.

Let $\Gamma_P\leq\Gamma$ be the subgroup of the mapping class group generated by Dehn twists along curves in $P$. In light of the previous paragraph, to prove the first part of the theorem it suffices to show that, for each of the finitely many points $t\in\Z^{3g+n-3}\cap M$ and $P\in\Pcal$, the set of nondegenerate integral points of $X_k(\Z)$ in the fiber $X_{k,t}^P=\tr_P^{-1}(t)$ decomposes into finitely many $\Gamma_P$-orbits. Now, it follows from \cite[Theorem 3.2]{whang4} that, if $X_{k,t}^P$ contains a nondegenerate integral point, then the fiber $X_{k,t}^P$ must be \emph{perfect} in the sense of \cite[Definition 3.1]{whang4}. It then follows by \cite[Corollary 4.4]{whang4} that $|\Gamma_P\backslash(X_k(\Z)\cap X_{k,t}^P)|<\infty$ which implies the desired result. Finally, the second assertion of the theorem follows from \cite[Theorem 1.1]{whang4}.
\end{proof}

It turns out that the fibers of the morphisms $\tr_P$ can be complicated in general, and descent along the fibers by Dehn twists can fail. Therefore, to give a full proof of Theorem \ref{theorem4}, we will modify our strategy to give an inductive argument on $(g,n)$; intuitively, our argument takes care of the action/angle variables one at a time.

We now describe our general strategy and the organization of this section. In Section \ref{sect:4.1}, we establish Theorem \ref{theorem4} in the cases $(g,n)=(1,1)$ and $(0,4)$ using a classical elementary argument, essentially due to Markoff. These will form the base case of our induction. In Section \ref{sect:4.2}, we establish two preliminary observations (Lemma \ref{compact} and Lemma \ref{curvelem}) to be used in the inductive step of our proof; we motivate these results by the following informal discussion. Let $K$ and $A$ be as in Theorem \ref{theorem4}, and suppose $\rho\in X_K(\Sigma,A)$ is a point which we would like to show to be $\Gamma(\Sigma)$-equivalent to a point in a fixed compact subset $L\subset X(\Sigma,\C)$. By Theorem \ref{theorem2}, there is an essential curve $a\subset\Sigma$ such that $|\tr\rho(a)|$ is bounded. The inductive hypothesis would grant us that, up to $\Gamma(\Sigma|a)$-action, we may assume that the image of $\rho$ under the restriction morphism
$$(-)|_{\Sigma|a}: X(\Sigma)\to X(\Sigma|a)$$
belongs to a compact subset of $X(\Sigma|a)$. At this stage, one wants the fiber $V$ of the above morphism containing $\rho$ to be a $1$-dimensional algebraic torus (and the group $\langle\tau_a\rangle$ generated by Dehn twist along $a$ to act cocompactly on $V(\C)$). However, it is possible for the dimension of $V$ to exceed $1$, if the restriction of $\rho$ is reducible on (a ``nontrivial'' component of) $\Sigma|a$. To avoid this, we show in Lemma \ref{compact} that one can choose the curve $a\subset\Sigma$ to satisfy favorable properties, bypassing the reducibility problem. In addition, to facilitate the study of Dehn twists $\langle\tau_a\rangle$ along fibers of $(-)|_{\Sigma|a}$, we prove (Lemma \ref{compact}) that certain bounded subsets on the character variety can be ``lifted'' to bounded subsets on the representation variety, where we may use the description of lifts of Dehn twists given in Section \ref{sect:2.2.3}. Having established these in Section \ref{sect:4.2}, in Section \ref{sect:4.3} we establish the general case of Theorem \ref{theorem4} by induction.

\subsection{Markoff descent}  \label{sect:4.1}
In this subsection, we establish Theorem \ref{theorem4} in the cases $(g,n)=(1,1)$ and $(0,4)$. Given $k=(k_1,\dots,k_n)\in\A^n(\C)$, we shall write
$$\Height(k)=\max\{1,|k_1|,\dots,|k_n|\}.$$
We remark that Lemmas \ref{lem11} and \ref{lem04} below give an elementary proof of Theorem \ref{theorem2} in the case where $\Sigma$ is a surface of type $(1,1)$ or $(0,4)$. The explicit form of the bounds obtained therein is not important so much as the existence of such bounds. Nevertheless, we remark in passing that these explicit bounds, together with quantitative refinements of Lemmas \ref{lem11b} and \ref{lem04b}, can be used together to give a different, more elementary proof of Corollary \ref{pantsco} for surfaces of low genus.

\subsubsection{Surfaces of type $(1,1)$} \label{sect:4.1.1}
Let $\Sigma$ be a surface of type $(1,1)$. By Section \ref{sect:2.3.1}, we have an identification of the moduli space $X_k=X_k(\Sigma)$ with the affine cubic algebraic surface in $\A_{x,y,z}^3$ given by the equation
$$x^2+y^2+z^2-xyz-2=k.$$
The mapping class group $\Gamma=\Gamma(\Sigma)$ acts on $X_k$ via polynomial transformations. Up to finite index, it coincides with the group $\Gamma'$ of automorphisms of $X_k$ generated by the transpositions and even sign changes of coordinates as well as the Vieta involutions of the form $(x,y,z)\mapsto(x,y,xy-z)$.

The mapping class group dynamics on $X_k(\R)$ was analyzed in detail by Goldman \cite{goldman4}, and the work of Ghosh-Sarnak \cite{gs} establishes a remarkable exact fundamental set for the action of $\Gamma'$ on the integral points $X_k(\Z)$ for admissible $k$. We emphasize that our focus is on establishing descent for the complex points; cf.~Silverman \cite{silverman}. Lemmas \ref{lem11} and \ref{lem11b} below establish Theorems \ref{theorem2} and \ref{theorem4} for $(g,n)=(1,1)$ by an elementary method.

\begin{lemma}
\label{lem11}
Let $\Sigma$ be a surface of type $(1,1)$. There is a constant $C>0$ independent of $k\in\C$ such that, given any $\rho\in X_k(\Sigma,\C)$, there exists some $\gamma\in\Gamma$ such that $\gamma^*\rho=(x,y,z)$ satisfies
$$\min\{|x|,|y|,|z|\}\leq C\cdot \Height(k)^{1/3}.$$
\end{lemma}

\begin{proof}
We notice that it suffices to prove the lemma with $\Gamma$ replaced by $\Gamma'$. Suppose first that $\rho=(x,y,z)$ satisfies
\begin{align*}|x|\leq|y|\leq|z|\leq|xy-z|.\tag{$*$}
\end{align*}
If $|x|\leq 8$ then we are done, so assume otherwise. Note that we have
$$|z|^2\leq|z(xy-z)|=|x^2+y^2-2-k|\leq 2|y|^2+2+|k|.$$
Suppose first that $|xy|\leq2|z|$. We then have
$$\frac{|xy|^2}{4}\leq|z|^2\leq|x^2+y^2-2-k|\leq2|y|^2+2+|k|.$$
Dividing both sides by $|y|^2$, we obtain
$$\frac{|x|^2}{4}\leq2+\frac{2+|k|}{|y|^2}\leq2+\frac{2+|k|}{|x|^2}\leq2\max\left\{2,\frac{2+|k|}{|x|^2}\right\}$$
and therefore $|x|^4\leq8(2+|k|)$, as desired. Suppose next that $|xy|\geq2|z|$. We have
$$\frac{|z||xy|}{2}\leq|z(xy-z)|\leq2|y|^2+2+|k|.$$
Dividing both sides by $|yz|$ and arguing as above, we obtain $|x|^3\leq 4(2+|k|)$, as desired. It therefore remains to consider the case where no $\gamma^*\rho$ with $\gamma\in\Gamma'$ satisfies the analogue of $(*)$. Assuming $|x|\leq|y|\leq|z|$ without loss of generality, we must have $|xy-z|\leq|y|$, since otherwise we may replace $(x,y,z)$ by $(x,y,z')$, $z'=xy-z$, satisfying an analogue of $(*)$. If $|x|\leq 12$ then we are done, so assume otherwise. Suppose first that $2|y|\geq|z|$. We then have
$$|xz^2|/2\leq|xyz|=|x^2+y^2+z^2-2-k|\leq 3|z|^2+2+|k|.$$
Dividing both sides by $|z|^2$ and arguing as before, we find that $|x|^3\leq4(2+|k|)$. Thus, assume that $2|y|\leq|z|$. We then replace $(x,y,z)$ by $(x,xy-z,y)$ and repeat the above argument. Since
$$2\max\{|x|,|y|,|xy-z|\}=2|y|\leq|z|=\max\{|x|,|y|,|z|\},$$
this cannot continue indefinitely without reaching our desired result.
\end{proof}

\begin{lemma}
\label{lem11b}
Let $\Sigma$ be a surface of type $(1,1)$. Let $K\subset\C$ be a compact set. Let $A$ be a set such that
\begin{enumerate}
	\item[\textup{(1)}] $A\subset\R$ and $\dist(A,\{\pm2\})>0$, or
	\item[\textup{(2)}] $A\subset\C$ and $\dist(A,[-2,2])>0$.
\end{enumerate}
There is a constant $C=C(K,A)>0$ such that, for any $\rho\in X_K(\Sigma,A)$, there exists $\gamma\in\Gamma$ so that $\gamma^*\rho=(x,y,z)$ satisfies $\max\{|x|,|y|,|z|\}\leq C$.
\end{lemma}

\begin{proof}
We notice that it suffices to prove the lemma with $\Gamma$ replaced by $\Gamma'$. Fix $\delta>0$ such that
$$|\lambda+\lambda^{-1}-x|>\delta$$
for every $x\in A$ and $\lambda\in\K$ with $1\leq|\lambda|\leq \sqrt{1+\delta}$, where $\K=\R$ if $A$ satisfies the first condition in the lemma and $\K=\C$ otherwise. Such a number $\delta$ exists by our assumption on $A$. Let $\rho=(x,y,z)\in X_k(A)$ be any point with $k\in K$. By Lemma \ref{lem11}, we may assume without loss of generality that
$$|x|\leq C_1(1+|k|)^{1/3}$$
for a positive constant $C_1$ independent of $\rho$ and $k$. Let $\epsilon=|x^2-2-k|$. Writing
$$c=\inf\{|y'|:(x,y',z')\sim(x,y,z)\}$$
where $\sim$ indicates equivalence under $\Gamma'$, we may further assume that
$$\min\{|y|,|z|\}\leq (c^2+\epsilon)^{1/2}$$
and that $|y|\leq|z|$. We claim that we can take $|y|^2\leq\epsilon/\delta$ up to $\Gamma'$-equivalence. Indeed, suppose otherwise. We have the following three cases:
\begin{enumerate}
	\item[\textup{(1)}] We have $|xy-z|\leq|y|\leq|z|$. Replacing $(x,y,z)$ by $(x,y',z')=(x,xy-z,y)$, we may thus assume $c\leq|y|\leq|z|\leq (c^2+\epsilon)^{1/2}$ and thus
	$$|z|^2\leq|y|^2+\epsilon\leq|y|^2(1+\delta).$$
	\item[\textup{(2)}] We have $|y|\leq|xy-z|\leq|z|$. Replacing $(x,y,z)$ by $(x,y',z')=(x,y,xy-z)$, we reduce to the third case below.
	\item[\textup{(3)}] We have $|y|\leq|z|\leq|xy-z|$. In particular, we have
	$$|z|^2\leq|z(xy-z)|=|x^2+y^2-2-k|\leq |y|^2+\epsilon<|y|^2(1+\delta).$$
\end{enumerate}
Thus, in all cases, the ratio $\lambda=z/y$ satisfies $1\leq|\lambda|\leq\sqrt{1+\delta}$. Dividing the equation $y^2+z^2-xyz=2+k-x^2$ by $yz$, we have
$$\left|\lambda^{-1}+\lambda-x\right|=\frac{\epsilon}{|yz|}\leq\delta,$$
contradicting the definition of $\delta$. Thus, we must have $|y|\leq\epsilon/\delta$. Finally, considering the above three cases again, we have
$$\max\{|y|,|z|\}\leq\epsilon(1+\delta)/\delta,$$
up to replacing $(x,y,z)$ by $(x,xy-z,y)$ in case $|xy-z|\leq|y|\leq|z|$. Combining this with the bound on $|x|$ in the beginning, we have proved our desired result.
\end{proof}

\subsubsection{Surfaces of type $(0,4)$} \label{sect:4.1.2}
Let $\Sigma$ be a surface of type $(0,4)$. By Section \ref{sect:2.3.2}, we have an identification of the moduli space $X_k=X_k(\Sigma)$ with the affine cubic algebraic surface in $\A_{x,y,z}^3$ given by the equation
$$x^2+y^2+z^2+xyz=ax+by+cz+d$$
with $a,b,c,d$ appropriately determined by $k$. The mapping class group $\Gamma=\Gamma(\Sigma)$ acts on $X_k$ via polynomial transformations. Let $\Gamma'$ be the group of automorphisms of $X_k$ generated by the Vieta involutions
\begin{align*}
&\tau_x^*:(x,y,z)\mapsto(a-yz-x,y,z),\\
&\tau_y^*:(x,y,z)\mapsto(x,b-xz-y,z),\\
&\tau_z^*:(x,y,z)\mapsto(x,y,c-xy-z).
\end{align*}
Two points $\rho,\rho'\in X_k(\C)$ are $\Gamma'$-equivalent if and only if they are $\Gamma$-equivalent or $\rho$ is $\Gamma$-equivalent to all of $\tau_x^*\rho'$, $\tau_y^*\rho'$, and $\tau_z^*\rho'$. The following lemmas prove Theorems \ref{theorem2} and \ref{theorem4} for $(g,n)=(0,4)$.

\begin{lemma}
\label{lem04}
Let $\Sigma$ be a surface of type $(0,4)$. There is a constant $C>0$ independent of $k\in\C^4$ such that, given any $\rho\in X_k(\C)$, there exists some $\gamma\in\Gamma$ such that $\gamma^*\rho=(x,y,z)$ satisfies one of:
\begin{enumerate}
	\item[\textup{(1)}] $\min\{|x|,|y|,|z|\}\leq C$,
	\item[\textup{(2)}] $|yz|\leq C\Height(a)$,
	\item[\textup{(3)}] $|xz|\leq C\Height(b)$,
	\item[\textup{(4)}] $|xy|\leq C\Height(c)$, or
	\item[\textup{(5)}] $|xyz|\leq C\Height(d)$.
\end{enumerate}
\end{lemma}

\begin{proof}
We proceed as in the proof of Lemma \ref{lem11}. Suppose first $\rho=(x,y,z)$ satisfies
\begin{align*}|x|\leq|y|\leq|z|\leq|c-xy-z|.\tag{$*$}
\end{align*}
Note that we have
$$|z|^2\leq|z(c-xy-z)|=|x^2+y^2-ax-by-d|\leq 4\max\{2|y|^2,|ax|,|by|,|d|\}.$$
Assume first that $|xy|\leq2|z|$. We then have
$$\frac{|xyz|}{2}\leq|z|^2\leq4\max\{2|y|^2,|ax|,|by|,|d|\}$$
which gives us the desired result. Assume next that $|xy|\geq2|z|$. If now $|xy|\leq 3|c|$, then we are done, we may assume $|xy|\geq 3|c|$. We then have
$$\frac{|xyz|}{6}\leq|z(c-xy-z)|\leq4\max\{2|y|^2,|ax|,|by|,|d|\}$$
as desired. If $\rho=(x,y,z)$ satisfies an analogue of $(*)$ with a suitable change in the roles of the variables $x,y,z$, then the same argument applies. It therefore remains to consider the case where no $\gamma^*\rho$ with $\gamma\in\Gamma'$ satisfies an analogue of $(*)$. Let us assume that $|x|\leq|y|\leq|z|$; the other cases will follow similarly. We must have
$$|c-xy-z|\leq|y|,$$ since otherwise we may replace $(x,y,z)$ by $(x,y,z')$, $z'=c-xy-z$, satisfying an analogue of $(*)$. Now, if $|x|\leq 8$ then we are done, so assume otherwise. Suppose first that $2|y|\geq|z|$. We then have
\begin{align*}
|xz^2|/2\leq|xyz|&=|x^2+y^2+z^2-ax-by-cz-d|\\
&\leq5\max\{3|z|^2,|ax|,|by|,|cz|,|d|\}
\end{align*}
which gives us the desired result. Thus, assume that $2|y|\leq|z|$. We then replace $(x,y,z)$ by $(x,y,c-xy-z)$ and repeat the above argument. Since
$$2\max\{|x|,|y|,|c-xy-z|\}=2|y|\leq|z|=\max\{|x|,|y|,|z|\},$$
this cannot continue indefinitely without reaching our desired result.
\end{proof}

\begin{lemma}
\label{lem04b}
Let $\Sigma$ be a surface of type $(0,4)$.
Let $K\subset\C^4$ be a compact set. Let $A$ be a set such that:
\begin{enumerate}
	\item[\textup{(1)}] $A\subset\R$ and $\dist(A,\{\pm2\})>0$, or
	\item[\textup{(2)}] $A\subset\C$ and $\dist(A,[-2,2])>0$.
\end{enumerate}
There is a constant $C=C(K,A)>0$ such that, for every $\rho\in X_K(\Sigma,A)$, there exists $\gamma\in\Gamma$ such that $\gamma^*\rho=(x,y,z)$ satisfies $\max\{|x|,|y|,|z|\}\leq C$.
\end{lemma}

\begin{proof}
We notice that it suffices to prove the lemma with $\Gamma$ replaced by $\Gamma'$. We argue as in the proof of Lemma \ref{lem11b}. Fix $\delta>0$ such that
$$|\lambda+\lambda^{-1}-(-x)|>\delta$$
for every $x\in A$ and $\lambda\in\K$ with $1\leq|\lambda|\leq \sqrt{1+\delta}$, where $\K=\R$ if $A$ satisfies the first condition in the lemma and $\K=\C$ otherwise. Such a number $\delta$ exists by our assumption on $A$. Let $\rho=(x,y,z)\in X_k(A)$ be any point with $k\in K$. By Lemma \ref{lem04}, we may assume without loss of generality that
$$\min\{|x|,|y|,|z|\}\leq C_1\Height(k)^{4/3}$$
for a positive constant $C_1$ independent of $\rho$ and $k$. Assume $|x|=\min\{|x|,|y|,|z|\}$; the other cases will follow similarly. Let $\epsilon=|x^2-ax-d|$. Writing
$$c=\inf\{|y'|,|z'|:(x,y',z')\sim(x,y,z)\}$$
where $\sim$ indicates equivalence under $\Gamma'$, we may further assume that
$$\min\{|y|,|z|\}\leq (c^2+\epsilon)^{1/2}.$$
We assume $|y|\leq|z|$; the other cases will follow similarly. We shall first bound the size of $|y|$. If $|y|\leq4|b|/\delta$ or $|y|\leq 4|c|/\delta$ or $|y|^2\leq2\epsilon/\delta$, then we are done, so suppose otherwise. We have the following three cases:
\begin{enumerate}
	\item[\textup{(1)}] We have $|c-xy-z|\leq|y|\leq|z|$. Replacing $(x,y,z)$ by
	$$(x,y,z')=(x,y,c-xy-z)$$
	and switching the roles of $y$ and $z$, we may assume $c\leq|y|\leq|z|\leq (c^2+\epsilon)^{1/2}$ and thus
	$$|z|^2\leq|y|^2+\epsilon\leq|y|^2(1+\delta).$$
	\item[\textup{(2)}] We have $|y|\leq|c-xy-z|\leq|z|$. Replacing $(x,y,z)$ by
	$$(x,y',z')=(x,y,c-xy-z)$$ we reduce to the third case below.
	\item[\textup{(3)}] We have $|y|\leq|z|\leq|c-xy-z|$. In particular, we have
	$$|z|^2\leq|z(c-xy-z)|=|x^2+y^2-ax-by-d|< |y|^2(1+\delta).$$
\end{enumerate}
Thus, in all cases, the ratio $\lambda=z/y$ satisfies $1\leq|\lambda|\leq\sqrt{1+\delta}$. Dividing the equation $y^2+z^2+xyz=ax+by+cz+d-x^2$ by $yz$, we have
$$\left|\lambda^{-1}+\lambda+x\right|=\frac{|by|+|cz|+\epsilon}{|yz|}\leq\delta,$$
contradicting the definition of $\delta$. Thus, we must have $|y|\leq 4\max\{|b|,|c|\}/\delta$ or $|y|\leq2\epsilon/\delta$. Finally, considering the above three cases again, we also obtain a bound on $|z|$ in terms of $|x|$ and $\Height(k)$. Combining this with the bound on $|x|$ in the beginning, we have proved our desired result.
\end{proof}

\subsection{Preliminary observations}  \label{sect:4.2}
We collect two lemmas needed in our proof of Theorem \ref{theorem4}. Let us say that a subset of a complex analytic variety $V$ is \emph{bounded} if its closure in $V$ is compact. For example, a subset $U\subset\SL_2(\C)$ is bounded if and only if the matrix entries of elements in $U$ are uniformly bounded.

\begin{lemma}
\label{compact}
Let $\Sigma$ be a surface of type $(g,n)$ with $n\geq1$. Let $\del\Sigma=c_1\sqcup \dots\sqcup c_n$. Let $(\alpha_1,\dots,\alpha_{2g+n})$ be an optimal sequence of generators for $\pi_1\Sigma$ as in Definition \ref{optimal}, with each loop $\alpha_{2g+i}$ corresponding to the boundary curve $c_{2g+i}$. Consider the set $K=K_1\times K_2\subset X(c_1\sqcup\dots\sqcup c_{n-1},\C)\times X(c_n,\C)=X(\del\Sigma,\C)\simeq\C^n$ where
\begin{enumerate}
	\item[\textup{(1)}] $K_1\subset\C^{n-1}$ is compact, and
	\item[\textup{(2)}] $K_2\subset\C\setminus\{\pm2\}$ is compact.
\end{enumerate}
Suppose $L\subset X_K(\Sigma,\C)$ is a subset bounded in $X(\Sigma,\C)$ such that every $\rho\in L$ is irreducible. There there exists a compact subset $M\subset\Hom(\pi_1\Sigma,\SL_2(\C))$ satisfying the following conditions:
\begin{itemize}
	\item $\pi^{-1}(L)\subseteq M^{\SL_2(\C)}=\{g\rho g^{-1}:\rho\in M,g\in\SL_2(\C)\}$ where $\pi$ is the usual projection $\Hom(\pi_1\Sigma,\SL_2(\C))\to X(\Sigma,\C)$, and
	\item every $\rho\in M$ satisfies 
	$$\rho(\alpha_{2g+n})=\begin{bmatrix}\lambda&0\\ 0 &\lambda^{-1}\end{bmatrix}$$
	for some $\lambda\in\{z\in\C:|z|>1\}\cup\{e^{\pi it}:0<t<1\}$.	
\end{itemize}
\end{lemma}

\begin{proof}
Let $\rho:\pi_1\Sigma\to\SL_2(\C)$ be a representation whose class in $X(\C)$ lies in $L$. This implies in particular that the traces $\tr\rho(\alpha_i)$, $\tr\rho(\alpha_i\alpha_j)$, and $\tr\rho(\alpha_i\alpha_j\alpha_k)$ (formed from the elements of the optimal sequence of generators) are bounded by a constant depending only on $L$. Up to global conjugation, we have
$$\rho(\alpha_{2g+n})=\begin{bmatrix} \lambda&0\\ 0 &\lambda^{-1}\end{bmatrix}$$
with $\lambda\in\{z\in\C:|z|>1\}\cup\{e^{\pi it}:0<t<1\}$, and there is a bounded set $\Lambda\subset\C$ with $\dist(\Lambda,\{\pm1\})>0$, depending only on $K_2$, such that $\lambda\in \Lambda$. Let us write
$$\rho(\alpha_i)=\begin{bmatrix}a_i & b_i\\ c_i & d_i\end{bmatrix}$$
for $i=1,\dots,2g+n-1$. Note that we have
$$\begin{bmatrix}a_i \\ d_i\end{bmatrix}=\frac{1}{\lambda^{-1}-\lambda}\begin{bmatrix}\lambda^{-1} &-1\\-\lambda &1\end{bmatrix}\begin{bmatrix}a_i+d_i\\\lambda a_i+\lambda^{-1}d_i\end{bmatrix}.$$
Since $(\tr\rho(\alpha_i),\tr\rho(\alpha_i\alpha_{2g+n}))=(a_i+d_i,\lambda a_i+\lambda^{-1}d_i)$ is bounded in terms of $L$, the above expression shows that $(a_i,d_i)$ lies in a bounded subset of $\C^2$ depending only on $\Lambda$ and $L$. Now, $\rho$ is irreducible by our hypothesis that its class in $X(\C)$ belongs to $L$, so that there exists at least one $\alpha_i$ such that $c_i\neq0$. Up to global conjugation of $\rho$ by a diagonal matrix in $\SL_2(\C)$, noting that
$$\begin{bmatrix}\mu & 0 \\ 0 & \mu^{-1}\end{bmatrix}\begin{bmatrix}a_i & b_i \\ c_i & d_i\end{bmatrix}\begin{bmatrix}\mu^{-1} & 0 \\ 0 & \mu\end{bmatrix}=\begin{bmatrix}a_i & \mu^2 b_i\\ \mu^{-2}c_i & d_i\end{bmatrix}$$
we can assume that $\max\{|c_i|:1\leq i\leq 2g+n-1\}=1$ and that $c_{i_0}=1$ for some $1\leq i_0\leq 2g+n-1$. The equation $a_{i_0}d_{i_0}-b_{i_0}c_{i_0}=1$ shows that $b_{i_0}$ is bounded in terms of $\Lambda$ and $L$. Next, for $i\neq i_0$, we have
$$\tr\rho(\alpha_{i_0}\alpha_{i})=a_{i_0}a_i+d_{i_0}d_i+b_{i_0}c_{i}+b_ic_{i_0}.$$
Since the left hand side as well as the first three terms on the right hand side are bounded in terms of $\Lambda$ and $L$, it follows that so is $b_i$. This shows that $\rho$ lies in a bounded subset of $\Hom(\pi_1\Sigma,\SL_2(\C))$. Taking $M$ to be the closure of the set of all $\rho\in\Hom(\pi_1\Sigma,\SL_2(\C))$ thus obtained, we see that $M$ is compact and has the requisite properties.
\end{proof}

Let $\Sigma$ be a surface of type $(g,n)$ with $3g+n-3>0$. Let $K\subset\C^n$ be a compact subset, and let $A$ be a set satisfying one of the following:
\begin{enumerate}
	\item[\textup{(1)}] $A\subset\R$ and $\dist(A,\{\pm2\})>0$, or
	\item[\textup{(2)}] $A\subset\C$ and $\dist(A,[-2,2])>0$.
\end{enumerate}
The following lemma is a refinement of Theorem \ref{theorem2} for points in $\rho\in X_K(\Sigma,A)$.

\begin{lemma}
\label{curvelem}
Assume $3g+n-2\geq2$. There is a constant $C=C(K,A)$ such that, for any irreducible local system $\rho\in X_K(A)$, there is an essential curve $a\subset\Sigma$ with $|\tr_a(\rho)|\leq C$ such that one of the following holds.
\begin{enumerate}
	\item[\textup{(1)}] We have $g\geq1$, the curve $a$ is nonseparating, and $\rho|(\Sigma|a)$ is irreducible.
	\item[\textup{(2)}] We have $g=0$, the curve $a$ is separating with $\Sigma|a=\Sigma_1\sqcup\Sigma_2$, where $\rho|\Sigma_1$ is irreducible and $\Sigma_2$ is of type $(0,3)$.
\end{enumerate}
\end{lemma}

\begin{proof}
In our proof below, we shall denote by $C(K,A)$ or $C_i(K,A)$ suitable constants that depend only on $K$ and $A$.

Assume first that $g\geq1$. By Corollary \ref{pantsco}, there is a pants decomposition $P$ of $\Sigma$ such that $|\tr\rho(a)|\leq C_1(K,A)$ for every component $a$ of $P$. At least one curve component of $P$ must be nonseparating; let us pick one and label it $a$. We are done if $\rho|(\Sigma|a)$ is irreducible, so let us assume otherwise. Let $\Sigma'$ be an embedded one holed torus in $\Sigma$ that contains $a$. Let $c$ denote the boundary curve of $\Sigma'$. Since $\tr\rho(c)\neq2$ by our hypothesis that $\pm2\notin A$ and that $\rho$ is irreducible, we see $\rho|\Sigma'$ is irreducible. Writing $\tr\rho(a)=\lambda+\lambda^{-1}$ with $\lambda\in\C$ and $|\lambda|\geq1$, we have
$$|\tr\rho(c)|=|\lambda^2+\lambda^{-2}|\leq C_2(K,A).$$
Here, the equality $\tr\rho(c)=\lambda^2+\lambda^{-2}$ follows from the hypothesis that $\rho|(\Sigma|a)$ is reducible. Using Lemma \ref{lem11b} and the notation therein, we may assume that the restriction $\rho|\Sigma'=(x,y,z)\in X(\Sigma')$ satisfies
$$\max\{|x|,|y|,|z|\}\leq C_3(K,A).$$
Let $(\alpha_1,\alpha_2,\gamma)$ be an optimal sequence of generators for $\pi_1\Sigma'$ (as in Definition \ref{optimal}), so that the essential curves $a=a_1$, $a_2$, and $a_3$ (lying in the free homotopy classes of $\alpha_1$, $\alpha_2$, and $\alpha_1\alpha_2$, respectively) correspond to the variables $x,y,z$ and $\gamma$ is homotopic to a parametrization of $\del\Sigma'$. It suffices to show that $\rho|(\Sigma'|a_i)$ is irreducible for some $i$, since this would imply $\rho|(\Sigma|a_i)$ is irreducible.

So suppose the contrary. Since $\tr\rho(\gamma)\neq\pm2$ by our hypothesis that $\pm2\notin A$, we may assume up to global conjugation that $\rho(\gamma)$ is diagonal and $\rho|\Sigma''$ is upper triangular where $\Sigma''$ is the surface of type $(g-1,n+1)$ such that $\Sigma|c=\Sigma'\sqcup\Sigma''$. Our hypothesis that each $\rho|(\Sigma|a_i)$ is reducible implies that the matrices $\rho(\alpha_1)$, $\rho(\alpha_2)$, and $\rho(\alpha_1\alpha_2)$ must be upper or lower triangular, and in fact they must be simultaneously upper triangular or simultaneously lower triangular. If $\rho(\alpha_1)$, $\rho(\alpha_2)$, and $\rho(\alpha_1\alpha_2)$ are simultaneously upper triangular, or if $\rho|\Sigma''$ is diagonal, then we conclude that $\rho$ is reducible; a contradiction. So let us assume that $\rho(\alpha)$ is upper triangular and non-diagonal for some $\alpha\in\pi_1\Sigma''$, and suppose $\rho(\alpha_1)$ is lower triangular and non-diagonal (the cases where $\rho(\alpha_2)$ or $\rho(\alpha_1\alpha_2)$ is non-diagonal will be similar). But $\rho(\gamma)$ is diagonal, so we contradict the hypothesis that $\rho|(\Sigma|a_1)$ is reducible. This completes the argument for Lemma \ref{curvelem} in the case $g\geq1$.

It remains to consider the case where $g=0$. We begin with the following claim.

\begin{nclaim}
There is an essential curve $b\subset\Sigma$ with $\Sigma|b=\Sigma'\sqcup\Sigma''$, where $\Sigma'$ is of type $(0,4)$, $\Sigma''$ is of type $(0,n-2)$, and $\rho|\Sigma'=(x,y,z)\in X(\Sigma',\C)$ satisfies
$$\max\{|x|,|y|,|z|\}\leq C(K,A)$$
up to $\Gamma(\Sigma')$-action.
\end{nclaim}

\begin{proof}[Proof of claim]
Let us begin with a pants decomposition $P=a_1\sqcup\dots\sqcup a_{3g+n-3}$ of $\Sigma$ such that $|\tr\rho(a_i)|\leq C_1(K,A)$ for every $a_i$ (such $P$ exists by Corollary \ref{pantsco}). Let $V(P)$ be the set of connected components of $\Sigma|P$. Each element of $V(P)$ is a surface of type $(0,3)$ and $\# V(P)=|\chi(\Sigma)|$. Let $G(P)$ be the simple graph with vertex set $V(P)$ where two distinct vertices in $V(P)$ are adjacent (i.e.~joined by an edge) if and only if their images (as components of $\Sigma|P$) in $\Sigma$ overlap, with their intersection being a curve component of $P$. One of two possibilities occur.

\begin{itemize}
	\item [(a)] There is a vertex of degree $1$ in $G(P)$ whose adjacent vertex has degree $2$.
	\item [(b)] For every vertex of degree $1$ in $G(P)$, the adjacent vertex has degree $3$.
\end{itemize}
If (a) occurs, we may choose $b$ to be the curve so that $\Sigma'$ is the union of the image of the components of $V(P)$ corresponding to said vertices. So suppose (b) occurs. In this case, there must be a vertex, say $\Sigma_A$, of degree $3$ in $G(P)$ such that two of its adjacent vertices, say $\Sigma_B$ and $\Sigma_C$, have degree $1$. Let $\Sigma_{AB}$ be the surface of type $(0,4)$ obtained as the union of images of $\Sigma_A$ and $\Sigma_B$ in $\Sigma$. Since the trace of $\rho$ along the curves in $\del\Sigma_{AB}$ are bounded in absolute value in terms of $K$ and $C_1(K,A)$, by Lemma \ref{lem04b} we can find an essential curve $b\subset\Sigma_{AB}$ such that
\begin{enumerate}
	\item $|\tr\rho(b)|\leq C_2(K,A)$, and
	\item $b$ is not $\Gamma(\Sigma_{AB})$-equivalent to curve in $P$ given as the intersection of the images of $\Sigma_A$ and $\Sigma_B$.
\end{enumerate}
See Figure \ref{fig4} for an illustration of the above. Writing $\Sigma|b=\Sigma'\sqcup\Sigma''$ with $\Sigma'$ of type $(0,4)$, and applying Lemma \ref{lem04b} to $\rho|\Sigma'$, we prove the claim.
\begin{figure}[ht]
    \centering
    \includegraphics{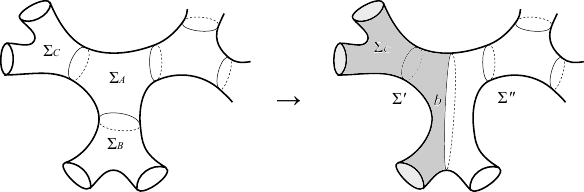}
    \caption{Construction of $\Sigma'$}
    \label{fig4}
\end{figure}
\end{proof}
To complete our proof of Lemma \ref{curvelem}, let $b$, $\Sigma'$, and $\Sigma''$ be as in the claim above. Let $a_1,a_2,a_3\subset\Sigma'$ be the essential curves corresponding to the variables $x,y,z$ in the claim. Writing $\Sigma|a_i=\Sigma_1^{i}\sqcup\Sigma_2^{i}$ with $\Sigma_1^{i}$ of type $(0,n-1)$ and $\Sigma_2^i$ of tye $(0,3)$, it suffices to show that $\rho|\Sigma_1^i$ is irreducible for some $i=1,2,3$.

So suppose toward contradiction that $\rho|\Sigma_1^a$ is reducible for every $i=1,2,3$. This implies in particular that $\rho|\Sigma''$ is reducible. Let $(\gamma_1,\gamma_2,\gamma_3,\gamma_4)$ be an optimal sequence of generators for $\pi_1\Sigma'$ so that $\gamma_1$ is homotopic to a parametrization of $b$ and $\gamma_1\gamma_2$, $\gamma_2\gamma_3$, and $\gamma_1\gamma_3$ are homotopic to parametrizations of $a_1$, $a_2$, and $a_3$ respectively. Up to global conjugation, we may assume that $\rho(\gamma_1)$ is diagonal (note that $\rho(\gamma_1)$ is diagonalizable since $\pm2\notin A$ by hypothesis) and that $\rho|\Sigma''$ is upper triangular.

Now, the hypotheses that $\rho|\Sigma_1^i$ are reducible imply that $\rho(\gamma_2)$, $\rho(\gamma_3)$, and $\rho(\gamma_4)$ are each upper or lower triangular. In fact, since $\gamma_1\gamma_2\gamma_3\gamma_4=1$, the matrices $\rho(\gamma_2)$, $\rho(\gamma_3)$, and $\rho(\gamma_4)$ must be simultaneously upper triangular or simultaneously lower triangular. If the former occurs then $\rho$ is reducible; a contradiction. So let us assume that $\rho(\gamma_2),\rho(\gamma_3)$ and $\rho(\gamma_4)$ are simultaneously lower triangular and one of them is non-diagonal, say $\rho(\gamma_2)$ (the other cases are similar). If $\rho|\Sigma''$ is diagonal, we again obtain a contradiction that $\rho$ is reducible. Thus, it remains to consider the case where $\rho(\alpha)$ is upper triangular and non-diagonal for some $\alpha\in\pi_1\Sigma''$, and $\rho(\gamma_2)$ is lower-triangular and non-diagonal. But $\rho(\gamma_1)$ is diagonal, so we contradict the hypothesis that $\rho|\Sigma_1^2$ is reducible (note that $\rho(\gamma_1)\neq\pm\mathbf1$ since $\pm2\notin A$). We have completed the proof of Lemma \ref{curvelem}.\end{proof}

\subsection{General case}  \label{sect:4.3}

In this subsection, we prove the general case of Theorem \ref{theorem4}. Let $\Sigma$ be a surface of type $(g,n)$ with $3g+n-3>0$. Let $K\subset\C^n$ be a compact subset, and let $A$ be a set satisfying one of the following:
\begin{enumerate}
	\item[\textup{(1)}] $A\subset\R$ and $\dist(A,\{\pm2\})>0$, or
	\item[\textup{(2)}] $A\subset\C$ and $\dist(A,[-2,2])>0$.
\end{enumerate}

\begin{thm}{1.4}
There is a compact set $L\subseteq X_K(\C)$ such that $X_K(A)\subseteq\Gamma\cdot L$.
\end{thm}

We will prove the theorem by induction on $3g+n-3\geq1$. The base case $3g+n-3=1$ where $(g,n)=(1,1)$ or $(0,4)$ was established in Section \ref{sect:4.1}. So let us assume $3g+n-3\geq2$. Below, we shall establish the inductive steps separately for the cases $g\geq1$ and $g=0$, in Sections \ref{sect:4.4.1} and \ref{sect:4.4.2} respectively. Afterwards, we conclude the proof of Theorem \ref{theorem4} in Section \ref{sect:4.4.3}.

\subsubsection{Case $g\geq1$}\label{sect:4.4.1}
We begin by noting that any point of $X_K(\Sigma,A)$ is automatically irreducible, since we have $2\notin A$ and $3g+n-3\geq2$ by assumption. Indeed, if a representation $\rho\in X(\Sigma,\C)$ is reducible then the boundary curve of any embedded one-holed torus in $\Sigma$ must have trace $2$ under $\rho$.

Let $a\subset\Sigma$ be a nonseparating curve. By Lemma \ref{curvelem}.(1) there exist compact subsets $C_\ell\subset\C\setminus[-2,2]$ and $C_e\subset(-2,2)$ (with subscripts $\ell$ and $e$ for \emph{loxodromic} and \emph{elliptic}) so that, for any $\rho\in X_K(\Sigma,A)$, there is $\gamma\in\Gamma(\Sigma)$ such that
\begin{itemize}
	\item[(a)] $\tr_a(\gamma^*\rho)\in C_\ell\sqcup C_e$ and
	\item[(b)] $(\gamma^*\rho)|(\Sigma|a)$ is irreducible.
\end{itemize}
Let us write $\del(\Sigma|a)=\del\Sigma\sqcup a_1\sqcup a_2$ where $a_1$ and $a_2$ are the boundary curves of $\del(\Sigma|a)$ corresponding to $a$, so that $X(\del(\Sigma|a),\C)\simeq X(\del\Sigma,\C)\times\C^2$ (see Section \ref{sect:2.2.2} for our notation). Let us write
\begin{align*}
&K_\ell=K\times\{(z,z):z\in C_\ell\}\subset X(\del(\Sigma|a),\C),\quad\text{and}\\
&K_e=K\times\{(z,z):z\in C_e\}\subset X(\del(\Sigma|a),\C).
\end{align*}
Thus, (a) implies $(\gamma^*\rho)|(\Sigma|a)\in X_{K_\ell}(\Sigma|a,A)\sqcup X_{K_e}(\Sigma|a,A)$. Now, $\Sigma|a$ is a surface of genus $g-1$ with $n+2$ boundary curves with $3(g-1)+(n+2)-3<3g+n-3$. By the inductive hypothesis, there exist compact sets
$$L_\ell\subseteq X_{K_\ell}(\Sigma|a,\C)\quad\text{and}\quad L_e\subseteq X_{K_e}(\Sigma|a,\C)$$
such that
$$X_{K_\ell}(\Sigma|a,A)\subseteq\Gamma(\Sigma|a)\cdot L_\ell\quad\text{and}\quad X_{K_e}(\Sigma|a,A)\subseteq\Gamma(\Sigma|a)\cdot L_e.$$
Let us define
\begin{align*}
&U=\{\rho\in X_K(\Sigma,A):\rho|(\Sigma|a)\in L_\ell\text{ irreducible}\},\quad\text{and}\\
&V=\{\rho\in X_K(\Sigma,A):\rho|(\Sigma|a)\in L_e\text{ irreducible}\}.
\end{align*}
Note that we have $X_K(\Sigma,A)\subseteq \Gamma(\Sigma)\cdot(U\cup V)$. To complete the inductive step for $g\geq1$ in our proof, it therefore suffices to prove the following.
\begin{nclaim}
The following holds.
\begin{enumerate}
	\item[(A)] There is a compact set $L\subset X_K(\Sigma,\C)$ with $U\subseteq\langle\tau_a\rangle\cdot L$, and
	\item[(B)] The set $V$ is bounded in $X(\Sigma,\R)$.
\end{enumerate}
\end{nclaim}

\begin{proof}(A) We will argue by passing to the representation variety $\Hom(\pi_1\Sigma,\SL_2(\C))$ and using the description of the lift of Dehn twist from Section \ref{sect:2.2.3}. We denote the standard projections from representation varieties to character varieties by $\pi$. Following the notation of the first part of Section \ref{sect:2.2.3}, let $\alpha$ be a loop parametrizing $a$, let $\beta$ be a simple loop on $\Sigma$ intersecting $\alpha$ exactly once, let $\alpha_1$ be a lift of $\alpha$ to $\Sigma|a$ parametrizing $a_1$, and let $\alpha_2$ be a loop on $\Sigma|a$ around $\alpha_2$ so that $\alpha_2=\beta^{-1}\alpha_1\beta$. We identify $\Hom(\pi_1\Sigma,\SL_2(\C))$ with its image under the closed immersion of complex varieties
$$\Hom(\pi_1\Sigma,\SL_2(\C))\subset\Hom(\pi_1(\Sigma|a),\SL_2(\C))\times\SL_2(\C),$$
given by $\rho\mapsto(\rho',B)=(\rho|_{\pi_1(\Sigma|a)},\rho(\beta))$, considered in the first part of Section \ref{sect:2.2.3}. Applying Lemma \ref{compact}, we can find a compact subset
$$M\subseteq\pi^{-1}(X_{K_\ell}(\Sigma|a,\C))$$
in the representation variety $\Hom(\pi_1(\Sigma|a),\SL_2(\C))$ such that
\begin{itemize}
	\item $\pi^{-1}(U)\subseteq (M\times\SL_2(\C))^{\SL_2(\C)}$, and
	\item every $\rho'\in M$ satisfies
$$\rho'(\alpha_1)=\begin{bmatrix}\lambda &0\\ 0 &\lambda^{-1}\end{bmatrix}\quad\text{for some $\lambda\in\C^*$ with $|\lambda|>1$}.$$
\end{itemize}
Here, $(M\times\SL_2(\C))^{\SL_2(\C)}$ indicates the union of the $\SL_2(\C)$-conjugation orbit of $M\times\SL_2(\C)$.
Let us consider an element
$$\rho=(\rho',B)\in\Hom(\pi_1\Sigma,\SL_2(\C))\cap (M\times\SL_2(\C)).$$
Writing
$$\rho'(\alpha_1)=\begin{bmatrix}\lambda & 0\\ 0 & \lambda^{-1}\end{bmatrix}\quad\text{and}\quad B=\begin{bmatrix}a&b\\c&d\end{bmatrix}\in\SL_2(\C),$$
we have
$$\rho'(\alpha_2)=B^{-1}\rho'(\alpha_1)B=\begin{bmatrix}\lambda+bc(\lambda-\lambda^{-1}) & bd(\lambda-\lambda^{-1})\\ -ac(\lambda-\lambda^{-1}) &\lambda^{-1}-bc(\lambda-\lambda^{-1})\end{bmatrix}.$$
Since $\rho'\in M$, the numbers $bc$, $bd$, $ac$, and $ad=bc+1$ belong to a compact subset of $\C$ depending only on $M$. Note that we cannot have $b=c=0$, since otherwise the boundary of the one-holed torus subsurface $\Sigma'\subset\Sigma$ whose fundamental group is generated by $\alpha_1$ and $\beta$ will have trace $2$, while $2\notin A$ by hypothesis. For each integer $k$, we have
$$\rho(\alpha_1)^kB=\begin{bmatrix}\lambda^k & 0\\ 0 & \lambda^{-k}\end{bmatrix}\begin{bmatrix}a & b\\ c &d\end{bmatrix}=\begin{bmatrix}a\lambda^k & b\lambda^k\\ c\lambda^{-k} &d\lambda^{-k}\end{bmatrix}=:\begin{bmatrix}a' & b'\\ c' & d'\end{bmatrix}.$$
We shall show in cases below that, for suitable $k$, the entries $a',b',c',d'$ of $\rho(\alpha_1)^k B$ will be bounded only in terms of $A$ and $M$.
\begin{enumerate}
	\item Suppose that $a=d=0$. Then since $bc=b'c'$ is bounded we see that for suitable $k$ the matrix $\rho(\alpha_1)^kB$ has entries lying in a compact subset of $\C$ depending only on $A$ and $M$.
	\item Suppose that $a$ or $d$ is nonzero. Then we must have $ab\neq0$ or $cd\neq0$ since $B$ is invertible and we cannot have $b=c=0$. We consider the case $ab\neq0$; the case $cd\neq0$ will follow by a similar argument. If $ab\neq0$, for suitable $k$ the matrix $\rho(\alpha_1)^kB$ above has entries such that $\max\{|a'|,|b'|\}$ belongs to a compact subset of $\C^\times$ depending only on $A$ and $M$. Also, $c'$ and $d'$ are bounded in terms of $K$ and $M$ by the boundedness of the products $ac=a'c'$, $ad=a'd'$, $bc=b'c'$, and $bd=b'd'$ remarked above.
\end{enumerate}
Using the description of the lift of Dehn twist in Section \ref{sect:2.2.3}, we conclude that there is a compact set $N\subset\SL_2(\C)$ depending only on $A$ and $M$ such that every
$$\rho=(\rho',B)\in\Hom(\pi_1\Sigma,\SL_2(\C))\cap (M\times\SL_2(\C))$$
is $\langle\tau_\alpha\rangle$-equivalent to a point of $M\times N$.
Writing
$$L=\pi(\Hom(\pi_1\Sigma,\SL_2(\C))\cap(M\times N))$$
we see that $U\subseteq\langle\tau_a\rangle\cdot L$
which gives the desired result.

(B) We may assume $A\subset\R$ with $\dist(A,\{\pm2\})>0$ since otherwise $V$ is empty by our hypothsis on $A$. Let $\alpha_1,\dots,\alpha_{2g+n}$ be an optimal sequence of generators for $\pi_1\Sigma$ (as in Definition \ref{optimal}) so that $\alpha_1$ parametrizes $a$. By Procesi's theorem (Fact \ref{fact}), the coordinate ring of $X(\Sigma)$ is generated by
$\tr_{\alpha_{i_1}\cdots\alpha_{i_k}}$
where $i_1<\dots<i_k$ are integers in $\{2,\dots,2g+n,1\}$ and $k\leq 3$. The values of these functions on $W$ are all bounded in terms of $A$, $K$, and $C_e$, except possibly for those of the form
$$\tr_{\alpha_2\alpha_1},\quad\tr_{\alpha_{2}\alpha_r},\quad\tr_{\alpha_2\alpha_r\alpha_1},\quad\text{and}\quad\tr_{\alpha_{2}\alpha_{s}\alpha_t}$$
for $3\leq r\leq 2g+n$ and $3\leq s<t\leq 2g+n$. We shall show that each of $\tr_{\alpha_2\alpha_r}$ above is bounded on $V$ in terms of $K$, $L_e$, and $A$; the same argument will apply equally to the functions of the form $\tr_{\alpha_2\alpha_s\alpha_t}$ above. Note that the simple loops (homotopic to) $\alpha_1$ and $\alpha_2\alpha_r$ together generate the fundamental group of an embedded subsurface $\Sigma'$ of type $(1,1)$ in $\Sigma$, with $\del\Sigma'\subset\Sigma|a$. Now, we have
$$\tr_{\alpha_1}^2+\tr_{\alpha_2\alpha_r}^2+\tr_{\alpha_1\alpha_2\alpha_r}^2-\tr_{\alpha_1}\tr_{\alpha_2\alpha_r}\tr_{\alpha_1\alpha_2\alpha_r}-2=\tr_{\del\Sigma'}.$$
Therefore, for each $\rho\in V$, the number $\tr\rho(\alpha_2\alpha_r)$ is the $X$- or $Y$-coordinate of a point $(X,Y)\in\R^2$ on an ellipse given by
$$X^2+Y^2-\tr\rho(\alpha_1)XY=\tr\rho(\del\Sigma')+2-\tr^2\rho(\alpha_1).$$
Thus, by elementary geometry, we have
$$|\tr\rho(\alpha_2\alpha_r)|\leq\sqrt{\frac{|\tr\rho(\del\Sigma')+2-\tr^2\rho(\alpha_1)||\tr\rho(\alpha_1)|^2}{|\tr^2\rho(\alpha_1)-4|}}.$$
The right hand side is bounded purely in terms of $K$, $C_e$, and $A$, noting in particular that the hypothesis $\dist(A,\{\pm2\})>0$ gives a lower bound on the denominator $|\tr^2\rho(\alpha_1)-4|$. This shows that the functions $\tr_{\alpha_2\alpha_r}$ (and similarly $\tr_{\alpha_2\alpha_s\alpha_t}$) above are bounded on $V$ in terms of $K$, $C_e$, and $A$. Thus $V$ is bounded, as desired.\end{proof}

\subsubsection{Case $g=0$} \label{sect:4.4.2}
The proof of the inductive step in this case is similar to the case $g\geq1$ above, with differences arising from the fact that we cut $\Sigma$ along separating curves. Since $\Sigma$ has genus zero, the set of reducible elements in $X_K(\Sigma,A)$ is bounded in $X(\Sigma,\C)$. To prove the claim in Theorem \ref{theorem4}, it thus suffices to consider only the irreducible elements in $X_K(\Sigma,A)$.

Considered up to mapping class group equivalence and isotopy, there are at most finitely many separating curves $a\subset\Sigma$ such that $\Sigma|a=\Sigma_1^a\sqcup\Sigma_2^a$ with $\Sigma_1^a$ of type $(0,n-1)$ and $\Sigma_2^a$ of type $(0,3)$. Let $\Scal$ be a complete set of representatives for these equivalence classes. Applying Lemma \ref{curvelem}.(2), we choose for each $a\in\Scal$ compact subsets $C_\ell^a\subset\C\setminus[-2,2]$ and $C_e^a\subset(-2,2)$ so that, for any irreducible $\rho\in X_K(\Sigma,A)$, there is some $\gamma\in\Gamma(\Sigma)$ and $a\in\Scal$ such that
\begin{itemize}
	\item[(a)] $\tr_a(\gamma^*\rho)\in C_\ell^a\sqcup C_e^a$ and
	\item[(b)] $(\gamma^*\rho)|\Sigma_1^a$ is irreducible.
\end{itemize}
For each $a\in\Scal$, let us write $\del\Sigma=B_1^a\sqcup B_2^a$ where $B_i^a=\Sigma_i^a\cap\del\Sigma$, so that we can write $\Sigma_i^a=B_i^a\sqcup a_i$ where $a_i$ is the boundary curve on $\Sigma_i^a$ corresponding to $a$. This means that we have $X(\del\Sigma_i^a)\simeq X(B_i^a)\times\C$ (see Section \ref{sect:2.2.2} for our notation). Let us define compact sets
$$K_{i,s}^a=\{\rho|B_i^a:\rho\in K\}\times C_s^a$$
for $i\in\{1,2\}$ and $s\in\{\ell,e\}$. Point (a) above implies that, for each $i=1,2$,
$$(\gamma^*\rho)|\Sigma_i^a\in X_{K_{i,\ell}^a}(\Sigma_i^a,A)\sqcup X_{K_{i,e}^a}(\Sigma_i^a,A).$$ By inductive hypothesis, for each $a\in\Scal$ there exist compact sets
$$L_{1,\ell}^a\subseteq X_{K_{1,\ell}^a}(\Sigma_1^a,\C)\quad\text{and}\quad L_{1,e}^a\subseteq X_{K_{1,e}^a}(\Sigma_1^a,\C)$$
such that
$$X_{K_{1,\ell}^a}(\Sigma_1^a,A)\subseteq\Gamma(\Sigma_1^a)\cdot L_{1,\ell}^a\quad\text{and}\quad X_{K_{1,e}^a}(\Sigma_1^a,A)\subseteq\Gamma(\Sigma_1^a)\cdot L_{1,e}^a.$$
Let us define
\begin{align*}
&U^a=\{\rho\in X_K(\Sigma,A):\rho|\Sigma_1^a\in L_{1,\ell}^a\text{ irreducible}\},\quad\text{and}\\
&V^a=\{\rho\in X_K(\Sigma,A):\rho|\Sigma_1^a\in L_{1,e}^a\text{ irreducible}\}.
\end{align*}
Note that $X_K(\Sigma,A)\subseteq\Gamma(\Sigma)\cdot(\bigcup_{a\in\Scal}U_a\cup V_a)$. To complete the inductive step for $g=0$ in our proof, it therefore suffices to prove the following.

\begin{nclaim}
Fix $a\in\Scal$. The following holds.
\begin{enumerate}
	\item [(A)] There is a compact set $L^a\subset X_K(\Sigma,\C)$ with $U^a\subseteq\langle\tau_a\rangle\cdot L^a$.
	\item [(B)] The set $V^a$ is bounded in $X(\Sigma,\R)$.
\end{enumerate}
\end{nclaim}

\begin{proof}(A)
We shall pass to the representation variety $\Hom(\pi_1\Sigma,\SL_2(\C))$, and use the description of the lift of Dehn twist from Section \ref{sect:2.2.3}. As in the second part of Section \ref{sect:2.2.3}, let $\alpha$ be a loop parametrizing $a$, and let $\alpha_i$ be a lift of $\alpha$ to $\Sigma_i^a$ parametrizing $\alpha_i$. Let us identify $\Hom(\pi_1\Sigma,\SL_2(\C))$ with its image under the closed immersion of complex varieties
$$\Hom(\pi_1\Sigma,\SL_2(\C))\subset\Hom(\pi_1\Sigma_1^a,\SL_2(\C))\times\Hom(\pi_1\Sigma_2^a,\SL_2(\C)),$$
given by $\rho\mapsto(\rho_1,\rho_2)=(\rho|_{\pi_1\Sigma_1^a},\rho|_{\pi_1\Sigma_2^a})$, considered in the second part of Section \ref{sect:2.2.3}. Applying Lemma \ref{compact}, we can find a compact subset
$$M^a\subseteq\pi^{-1}(X_{K_{1,\ell}^a}(\Sigma_1^a,\C))\subseteq\Hom(\pi_1\Sigma_1^a,\SL_2(\C))$$
in the representation variety of $\Sigma_1^a$ such that
\begin{itemize}
	\item $\pi^{-1}(U^a)\subseteq(M^a\times\pi^{-1}(X_{K_{2,\ell}^a}(\Sigma_2^a,\C)))^{\SL_2(\C)}$, and
	\item every $\rho_1\in M^a$ satisfies
$$\rho_1(\alpha_1)=\begin{bmatrix}\lambda &0\\ 0 &\lambda^{-1}\end{bmatrix}\quad\text{for some $\lambda\in\C^*$ with $|\lambda|>1$}.$$
\end{itemize}
Let us consider an element
$$\rho=(\rho_1,\rho_2)\in\Hom(\pi_1\Sigma,\SL_2(\C))\cap(M^a\times \pi^{-1}(X_{K_{2,\ell}^a}(\Sigma,\C)).$$
Let $\alpha_2=\gamma_1,\gamma_2,\gamma_3$ be an optimal sequence of generators for $\pi_1\Sigma_2^a$. Let us write
$$\rho(\gamma_1)=\begin{bmatrix}\lambda & 0 \\ 0 &\lambda^{-1}\end{bmatrix}\quad\text{and}\quad
\rho(\gamma_2)=\begin{bmatrix}a & b\\ c & d\end{bmatrix}.$$
Note that we have
$$\begin{bmatrix}a \\ d\end{bmatrix}=\frac{1}{\lambda^{-1}-\lambda}\begin{bmatrix}\lambda^{-1} &-1\\-\lambda &1\end{bmatrix}\begin{bmatrix}a+d\\\lambda a+\lambda^{-1}d\end{bmatrix}.$$
Since $(\tr\rho(\gamma_2),\tr\rho(\gamma_3))=(a+d,\lambda a+\lambda^{-1}d)$ is bounded in terms of $K$, the above expression shows that $(a,d)$ lies in a compact subset of $\C^2$ depending only on $K_{2,\ell}^a$ and $A$. In addition, $bc=ad-1$ is bounded in terms of $K_{2,\ell}^a$ and $A$. Now, note that
$$\rho(\gamma_1)^k\rho(\gamma_2)\rho(\gamma_1)^{-k}=\begin{bmatrix} a & b' \\ c' & d\end{bmatrix}=\begin{bmatrix} a & b\lambda^{2k} \\ c\lambda^{-2k} & d\end{bmatrix}$$
for any integer $k$. Thus, for suitable $k$, we see that $c'$ and $d'$ are bounded in terms of $K_{2,\ell}^a$ and $A$. Using the description of the lift of Dehn twist in Section \ref{sect:2.2.3}, we conclude that there is a compact set $N^a\subset\pi^{-1}(X_{K_{2,\ell}^a}(\Sigma,\C)$ depending only on $K_{2,\ell}^a$ and $A$ such that every
$$\rho=(\rho_1,\rho_2)\in\Hom(\pi_1\Sigma,\SL_2(\C))\cap(M^a\times \pi^{-1}(X_{K_{2,\ell}^a}(\Sigma,\C))$$
is $\langle\tau_\alpha\rangle$-equivalent to a point of $M^a\times N^a$. Writing
$$L^a=\pi(\Hom(\pi_1\Sigma,\SL_2(\C))\cap (M^a\times N^a))$$
we see that $U^a\subseteq\langle\tau_a\rangle\cdot L^a$ which gives the desired result.

(B) We may assume $A\subset\R$ with $\dist(A,\{\pm2\})>0$ since otherwise $V$ must be empty by our hypothsis on $A$. Let $\alpha_1,\dots,\alpha_{n}$ be an optimal sequence of generators for $\pi_1\Sigma$ so that $a$ lies in the free homotopy class of $\alpha_1\alpha_2$. By Procesi's theorem (Fact \ref{fact}), the coordinate ring of $X(\Sigma)$ is generated by
$\tr_{\alpha_{i_1}\cdots\alpha_{i_k}}$
where $i_1<\dots<i_k$ are integers in $\{2,\dots,n\}$ and $k\leq 3$. The values of these functions on $W$ are all bounded in terms of $K$, $L_{1,e}^a$, and $A$, except possibly for those of the form
$$\tr_{\alpha_{2}\alpha_r}\quad\text{and}\quad\tr_{\alpha_{2}\alpha_{s}\alpha_t}$$
for $3\leq r\leq n$ and $3\leq s<t\leq n$. Now, we shall show that each of $\tr_{\alpha_2\alpha_r}$ above is bounded on $W$ in terms of $K$, $L_{1,e}^a$, and $A$; a similar argument will apply equally to the functions of the form $\tr_{\alpha_2\alpha_s\alpha_t}$ above. For simplicity, up to choosing a new optimal sequence of generators we may assume without loss of generality that $r=3$. Consider the subsurface $\Sigma'$ of type $(0,4)$ in $\Sigma$ whose fundamental group is generated by $\alpha_1,\alpha_2,\alpha_3$. For each $\rho\in V$, the number $\tr\rho(\alpha_2\alpha_r)$ is a coordinate of a point $(X,Y)$ on a (possibly degenerate) ellipse given by
$$X^2+Y^2+\tr\rho(a) XY+q_1X+q_2Y=q_3$$
for some $q_i\in\R$ which are bounded in terms of $A, K_{1,e}^a$, and $K$. (This is seen by considering the equation for the relative character variety of $X(\Sigma')$ seen in Section \ref{sect:2.3.2}.) Moreover, $\tr\rho(a)$ is bounded in terms of $A$ and $K_{1,e}^a$ away from $\{\pm2\}$. As before, by elementary geometry, we deduce that $\tr\rho(\alpha_2\alpha_3)$ is bounded in terms of $A,K$. This shows that the functions $\tr_{\alpha_2\alpha_r}$ (and similarly $\tr_{\alpha_2\alpha_s\alpha_t}$) above are bounded on $V$ in terms of $K$, and $A$. This shows that $V$ is bounded, as desired.
\end{proof}

\subsubsection{Conclusion of proof of Theorem \ref{theorem4}}\label{sect:4.4.3}
In Section \ref{sect:4.1}, we proved Theorem \ref{theorem4} in the cases $(g,n)=(1,1)$ and $(0,4)$. For $g\geq1$ with $(g,n)\neq(1,1)$, in Section \ref{sect:4.4.1} we established the case $(g,n)$ of Theorem \ref{theorem4} under the assumption of the case $(g-1,n+2)$. For $g=0$ and $n\geq5$, in Section \ref{sect:4.4.2} we established the case $(0,n)$ of Theorem \ref{theorem4} under the assumption of the case $(0,n-1)$. Combined, this completes the proof of Theorem \ref{theorem4} for all cases $(g,n)$ with $3g+n-3>0$.

\subsection{Application} Let $\Sigma$ be a surface of type $(g,n)$ with $3g+n-3>0$, and let $A$ and $K$ be as in the hypothesis of Theorem \ref{theorem4}. The following corollary of Theorem \ref{theorem4} gives a generalization of McKean's finiteness theorem \cite{mckean} for length isospectral families of closed Riemann surfaces. It implies in particular that, for each $k\in\Z$, the set $X_k(\Sigma,\Z\setminus\{\pm2\})$ consists of finitely many mapping class group orbits; this is a strengthening of the first part of our main Diophantine result, Theorem \ref{theorem1}.

\begin{corollary}
\label{finiteness}
Assume that $A$ is discrete in $\R$ or $\C$. Then $X_K(A)$ consists of finitely many mapping class group orbits.
\end{corollary}

\begin{proof}
We first note that the coordinate ring of $X$ is generated by the trace functions of finitely many nondegenerate (simple closed) curves on $\Sigma$; this follows by applying Procesi's theorem (Fact \ref{fact}) with a choice of optimal sequence of generators for $\pi_1\Sigma$ (Definition \ref{optimal}). In particular, for each $k\in X(\del\Sigma,\C)$ the set $X_k(A)$ is a closed discrete subset of $X_k(\C)$ (see Section \ref{sect:2.2.2} for our notation $X(\del\Sigma,\C)$). Combining this with the compactness property from Theorem \ref{theorem4}, we see that $X_k(A)$ consists of only finitely many mapping class group orbits for any $k\in X(\del\Sigma,\C)$. (We remark that this is sufficient for Theorem \ref{theorem1}.) It only remains to prove the following.

\begin{nclaim}
There are at most finitely many $k\in K$ for which $X_k(A)$ is nonempty.
\end{nclaim}

Let us choose a compact subset $L\subset X(\C)$ satisfying the conclusion of Theorem \ref{theorem4}, so that $X_k(A)$ is nonempty if and only if $X_k(A)\cap L$ is nonempty. We need to consider different cases; let us first suppose that $(g,n)=(1,1)$. For any $k\in K$, we have the presentation
$$x^2+y^2+z^2-xyz-2=k$$
of $X_k$ as given in Section \ref{sect:2.3}. The image of $X_K(A)\cap L$ under the morphism $(x,y,z)$ to $\A^3$ is compact and discrete, and hence finite. In particular, the above equation shows that there are only finitely many possibilities for $k\in K$ such that $X_K(A)\cap L$ is nonempty, as desired.

Suppose next that $(g,n)\neq(1,1)$. A pants decomposition of $\Sigma$ shows that, for each boundary curve $c$ of $\Sigma$, there is an immersion $i:\Sigma'\to\Sigma$ of a surface $\Sigma'$ of type $(0,4)$ which is an embedding on the interior of $\Sigma'$ and sends a boundary curve of $\Sigma'$ to $c$. In particular, there exists a compact subset $K'\subset X(\del\Sigma',\C)$ depending only on $K$, $A$, and $L$ such that $i^*(\rho)\in X_{K'}(\Sigma',A)$ for every $\rho\in X_K(A)\cap L$; hence, we are reduced to proving the claim for surfaces of type $(0,4)$. So suppose that $\Sigma$ is of type $(0,4)$. For each $k\in K$, we have a presentation
$$x^2+y^2+z^2+xyz=ax+by+cz+d$$
of $X_k$ as given in Section \ref{sect:2.3}, where the coefficients $a,b,c,d$ are suitable functions of $k$. The images of $X_K(A)\cap L$ under the regular functions $x$, $y$, $z$ are each compact and discrete and hence finite. The same holds for the image of $X_K(A)\cap L$ under the regular function $a-yz-x$ (note that the value of $a-yz-x$ must lie in $A$ for any $\rho\in X(A)$, as seen by the explicit expression for Dehn twists given in Section \ref{sect:2.3}). Therefore, it follows that there are at most finitely many possible values of $a$ for elements of $X_K(A)\cap L$. A similar conclusion holds for $b$ and $c$, and therefore for $d$ as well by the above equation. We are thus reduced to showing that the polynomial map $\Phi=(a,b,c,d):\C^4\to\C^4$ given in coordinate functions $k_1,\dots,k_4$ of the domain by
$$\left\{\begin{array}{l}a=k_1k_2+k_3k_4\\b=k_1k_4+k_2k_3\\c=k_1k_3+k_2k_4\end{array}\right.\quad\text{and}\quad d=4-\sum_{i=1}^4k_i^2-\prod_{i=1}^4k_i$$
has finite fibers everywhere.

This can be seen, for example, as follows. We first note that the fiber of $\Phi$ above any point is the set of complex points of an affine algebraic variety, and hence must be either finite or uncountable. On the other hand, defining $\Psi:\C^4\to\C^4$ to be the function $\Psi(\theta_1,\dots,\theta_4)=(2\cos\pi\theta_1,\dots,2\cos\pi\theta_4)$, a result of Lisovyy-Tykhyy \cite[Proposition 10]{lt} shows that the set of all points in the fiber of $\Phi\circ\Psi$ over any given point lie in the single orbit of a countable group. \emph{A fortiori}, the fiber of $\Phi$ over any point is countable, and hence finite (being the set of complex points of an affine algebraic variety).
\end{proof}

\section{Further remarks}
\label{sect:5}

\subsection{Arithmetic hyperbolic surfaces}
\label{sect:5.1}
Below, we shall briefly discuss the relationship between the Diophantine geometry of the moduli spaces $X_k$ and arithmetic hyperbolic surfaces. First, in Section \ref{sect:5.1.1}, we shall recall some basic results on arithmetic hyperbolic surfaces, including a characterization of arithmetic lattices in $\PSL_2(\R)$ due to Takeuchi \cite{takeuchi} and a result on their finitude (for bounded covolume) due to Borel \cite{borel} and Takeuchi \cite{takeuchi2}. In Section \ref{sect:5.1.2}, we mention specific consequences for the integral points $X_k(\Z)$ of the moduli space for suitable values of $k$, and discuss the classical Markoff example where $(g,n)=(1,1)$.

\subsubsection{Arithmetic lattices and their finitude}\label{sect:5.1.1}

We first recall the characterization of arithmetic lattices in $\PSL_2(\R)$. A \emph{lattice} $L\leq\PSL_2(\R)$ is a discrete subgroup of finite covolume. Let $k\subset\R$ be a totally real number field of degree $d$ over $\Q$. Let $i_1,\dots,i_d:k\hookrightarrow\R$ be the embeddings of $k$ such that $i_1={\id}$. Let $Q$ be a quaternion algebra over $k$ which is split at $i_1$ and ramified at all the other archimedean places of $k$. The splitting induces an embedding $\rho:Q\to M_2(\R)$. Let $\Ocal$ be an order in $Q$, and let $\Ocal^1$ be the group of norm one elements in $\Ocal$. The image of $\rho(\Ocal^1)\leq\SL_2(\R)$ under the projection $\SL_2(\R)\to\PSL_2(\R)$, which we shall denote $L(Q,\Ocal)$, is a lattice in $\PSL_2(\R)$.

\begin{definition}
A lattice $L\leq\PSL_2(\R)$ is \emph{arithmetic} if it is commensurable with some $L(Q,\Ocal)$ as above, and is \emph{derived from a quaternion algebra} if it is moreover of finite index in some $L(Q,\Ocal)$.
\end{definition}

Takeuchi \cite{takeuchi} gave a characterization of arithmetic lattices in $\PSL_2(\R)$ as follows. Given a subset $M\subset\PSL_2(\R)$, we shall denote $\tr(M)=\{\pm\tr(a):a\in M\}$, where $\tr(a)$ is the trace of $a$ which is well-defined up to sign.

\begin{theorem}[Takeuchi]
\label{tak}
A lattice $L\leq\PSL_2(\R)$ is arithmetic (up to conjugation) if and only if it satisfies the conditions $(\textup{I})$ and $(\textup{II}_1)$ below, and is moreover derived from a quaternion algebra if and only if it satisfies $(\textup{I})$ and $(\textup{II}_2)$.
\begin{enumerate}
	\item[$(\textup{I})$] The \emph{trace field} $K_L=\Q(\tr(L))\subset\R$ has finite degree over $\Q$, and $\tr(L)$ is contained in the ring of integers of $K_L$.
	\item[$(\textup{II}_1)$] For any field embedding $i:K_L\to\C$ which is not the identity on the subfield $K_{L}^{(2)}\subseteq K_L$ generated by $\{\tr(\gamma)^2:\gamma\in L\}$, the set $i(\tr(L))$ is bounded in $\C$.
		\item[$(\textup{II}_2)$] For any field embedding $i:K_L\to\C$ which is not the identity, the set $i(\tr(L))$ is bounded in $\C$.
\end{enumerate}
\end{theorem}

One can show that $K_L^{(2)}=K_{L^{(2)}}$ where $L^{(2)}$ is the finite index subgroup of $L$ generated by $\{\gamma^2:\gamma\in L\}$ (see \cite{takeuchi}). It follows that a lattice $L\leq\PSL_2(\R)$ satisfying condition $(\textup{I})$ above is arithmetic if and only if $L^{(2)}$ is derived from a quaternion algebra. We note that an arithmetic lattice $L\leq\PSL_2(\R)$ is non-cocompact if and only if $L^{(2)}$ is conjugate to a subgroup of $\PSL_2(\Z)$, or equivalently $K_{L}^{(2)}=\Q$.

Quotients of the hyperbolic plane $\H^2$ by arithmetic lattices in $\PSL_2(\R)$ are referred to as \emph{arithmetic hyperbolic surfaces}. In general, the quotient $S_L=L\backslash\H^2$ of $\H^2$ by a lattice $L\leq\PSL_2(\R)$ is a complete hyperbolic surface of finite area, having finitely many cusps and orbifold singularities. By classical works of Fricke and Klein, the lattice $L$ is generated by hyperbolic elements $\alpha_1,\beta_1,\dots,\alpha_g,\beta_g$ and elliptic or parabolic elements $\gamma_1,\dots,\gamma_n$ with relations
\begin{align}
\label{latticepres}
[\alpha_1,\beta_1]\cdots[\alpha_g,\beta_g]\gamma_1\cdots\gamma_n=1,\quad\gamma_i^{e_i}=1\quad(i=1,\dots,n)
\end{align}
where $2\leq e_1\leq\cdots\leq e_n\leq\infty$ (with $e_i=\infty$ understood as saying that $\gamma_i$ is parabolic). The sequence $(g;e_1,\dots,e_n)$ is referred to as the \emph{signature} of the lattice $L$ or the hyperbolic surface $S_L$, and we have
$$\frac{\Vol(L\backslash\H^2)}{2\pi}=2g-2+\sum_{i=1}^n\left(1-\frac{1}{e_i}\right).$$
Let $\Sigma$ be a surface of genus $g$ with $n$ punctures, and let $X=X(\Sigma)$ be the moduli space of $\SL_2$-local systems on $\Sigma$. A diffeomorphism of $\Sigma$ onto the complement of orbifold points on a hyperbolic surface $S$ of signature $(g;e_1,\dots,e_n)$ gives rise to a conjugacy class of a representation $\pi_1\Sigma\to\PSL_2(\R)$, whose image is a lattice $L\leq\PSL_2(\R)$ with $S\simeq S_L$. Such a representation admits finitely many lifts to $\SL_2(\R)$ up to conjugacy, and we shall denote by $T(g;e_1,\dots,e_n)\subset X(\R)$ the set of representations so obtained over all hyperbolic surfaces $S$ of signature $(g;e_1,\dots,e_n)$. Each $T(g;e_1,\dots,e_n)$ is contained in a finite union of relative moduli spaces $X_k(\R)$, where the possible traces
$$k\in\{2\cos(2\pi r):r\in\Q\}^n\subseteq\A^n(\R)$$
are determined by $(g;e_1,\dots,e_n)$. Each connected component of $T(g;e_1,\dots,e_n)$ is isomorphic to the Teichm\"uller space of marked hyperbolic (orbifold) structures on $\Sigma$ with signature $(g;e_1,\dots,e_n)$, on which the mapping class group descent is classical. We shall show how our arguments may be used to derive a weak finiteness statement (Corollary \ref{wearfi}) for arithmetic lattices in $\PSL_2(\R)$, which is superseded by a stronger result of Borel \cite{borel} and Takeuchi \cite{takeuchi2} discussed further below.

\begin{proposition}
For $M\geq0$ and $D\geq0$, there are only finitely many isomorphism classes of lattices $L\leq\PSL_2(\R)$ derived from a quaternion algebra having covolume $\Vol(L\backslash\H^2)\leq M$ and trace field degree $[K_L:\Q]\leq D$.
\end{proposition}

\begin{proof}
First, given $D\geq0$ and $N\geq0$, there are at most finitely many algebraic integers of degree at most $D$ all of whose conjugates have absolute value at most $N$. It follows that there is a closed discrete subset $A_D\subset\R$ such that we have $\tr(L)\subseteq A_D$ for every arithmetic lattice $L\leq\PSL_2(\R)$ derived from a quaternion algebra with trace field degree $[K_L:\Q]\leq D$. A bound on the volume $\Vol(L\backslash\H^2)$ shows that there are at most finitely many possibilities $(g;e_1,\dots,e_n)$ for the signature. For each choice of signature, the set $T(g;e_1,\dots,e_n)\cap X(A)$ consists of at most finitely mapping class group orbits. The desired result follows.
\end{proof}

\begin{corollary}
\label{wearfi}
For $M\geq0$ and $D\geq0$, there are only finitely many isomorphism classes of arithmetic lattices $L\leq \PSL_2(\R)$ of covolume $\Vol(L\backslash\H^2)\leq M$ which are
\begin{enumerate}
	\item[\textup{(1)}] non-cocompact, or
	\item[\textup{(2)}] cocompact with trace field degree $[K_L:\Q]\leq D$.
\end{enumerate}
\end{corollary}

\begin{proof}
Given an upper bound on the covolume $\Vol(L\backslash\H^2)$, there are only finitely many possibilities for the signature of $L$. The index $[L:L^{(2)}]$ depends only on the signature of $L$. The previous proposition shows that there are at most finitely choices of $L^{(2)}$ up to conjugation, recalling that $K_{L}^{(2)}=\Q$ if $L$ is non-cocompact. As noted in \cite{takeuchi2}, for each choice of $L^{(2)}$ the normalizer $N_{L^{(2)}}$ of $L^{(2)}$ in $\PSL_2(\R)$ contains $L^{(2)}$ with finite index. Since $L\leq N_{L^{(2)}}$, there are at most finitely choices of $L$.
\end{proof}

An analogous result is obtained in \cite[Theorem 4.3]{bmr} in the context of arithmetic hyperbolic surface bundles. For cocompact lattices, the finiteness result of Borel \cite{borel} and Takeuchi \cite{takeuchi2} below removes the assumption on the trace field degree.

\begin{ntheorem}[\cite{borel},\cite{takeuchi2}]
The collection of volumes of arithmetic hyperbolic surfaces is closed and discrete in $\R$ with finite multiplicities.
\end{ntheorem}

We remark that, given a totally real number field $k\subset\R$ with ring of integers $\Ocal_k$, it is not necessarily true that a point of $T(g;e_1,\dots,e_n)\cap X(\Ocal_k)$ is associated to an arithmetic hyperbolic surface. This amounts to the observation that there are lattices in $\PSL_2(\R)$ satisfying condition $(\textup{I})$ of Theorem \ref{tak} but failing $(\textup{II}_1)$; examples are given in \cite{takeuchi2}. Thus, the above theorem does not imply the finitude of mapping class group orbits on $T(g;e_1,\dots,e_n)\cap X(\Ocal_k)$, except in the case $k=\Q$ which we shall discuss in further detail.

\subsubsection{Relation to integral points}\label{sect:5.1.2}
We now restrict our discussion to the integral points $T(g;e_1,\dots,e_n)\cap X(\Z)$. We shall say that the signature $(g;e_1,\dots,e_n)$ is of \emph{compact type} if each $e_i$ is finite (i.e.~the associated arithmetic hyperbolic surface is compact), and is of \emph{noncompact type} otherwise.

If $(g;e_1,\dots,e_n)$ is of noncompact type, then each point of $T(g;e_1,\dots,e_n)\cap X(\Z)$ corresponds to a representation $\pi_1\Sigma\to\SL_2(\R)$ whose image in $\PSL_2(\R)$ is conjugate to a finite index subgroup of $\PSL_2(\Z)$. As mentioned in \cite{lubotzky}, on combining results of Newman \cite{newman} and Moser-Wyman \cite{mw} one finds that the number $a_N$ of index $N$ subgroups in $\PSL_2(\Z)$ is asymptotically
$$a_N\sim(12\pi\sqrt{e})^{-1/2}\exp\left(\frac{N\log N}{6}-\frac{N}{6}+N^{1/2}+N^{1/3}+\frac{\log N}{2}\right).$$
In particular, it follows that, on the moduli spaces $X_k(\Sigma)$ as $\Sigma$ varies in type $(g,n)$ with $n\geq1$ and $k\in\{0,\pm1,\pm2\}^n$ with at least one component of $k$ equal to $\pm2$, the average number of nondegenerate integral orbits grows exponentially in $|\chi(\Sigma)|$. This is not \emph{a priori} obvious without the moduli interpretation of $X_k(\Sigma)$.

If $(g;e_1,\dots,e_n)$ is of compact type, then a point of $T(g;e_1,\dots,e_n)\cap X(\Z)$ corresponds to a representation $\pi_1\Sigma\to\SL_2(\R)$ whose image is conjugate to a finite index subgroup of the group $\Ocal^1$ of norm one elements in an order $\Ocal$ of a quaternion algebra over $\Q$ which is split at $\infty$. This suggests that the arithmetic of quaternion algebras and their orders may be used to study integral points of $X$. Here, we give an instance of this analysis for the classical Markoff example $(g,n)=(1,1)$. As discussed in Section \ref{sect:2}, the moduli space $X_k(\Sigma)$ in this case is an affine cubic algebraic surface with equation
$$x^2+y^2+z^2-xyz-2=k,\quad k\in\Z.$$
The geometry and mapping class group dynamics on the set of real points $X_k(\R)$ of the moduli space is described in \cite{goldman4}. To discuss the point of contact between arithmetic hyperbolic surfaces and integral points, we shall consider $k\in\{-2,-1,0,-1\}$. For each such $k$, Goldman \cite{goldman4} observed that $X_k(\R)$ consists of five connected components, each of which is mapping class group invariant: one component is compact and consists of unitary representations $\pi_1\Sigma\to\SU(2)$ having boundary trace $k$, while the remaining four components comprise the locus $T(1;e)$ for a suitable signature $(1;e)$ depending on $k$.

Suppose first that $k\in\{0,\pm1\}$, so that the signature $(1;e)$ is of compact type. We shall show that $T(1;e)\cap X_k(\Z)$ is empty. Suppose otherwise; given a representation $\rho:\pi_1\Sigma\to\SL_2(\R)$ with class in $T(1;e)\cap X_k(\Z)$, as observed in \cite{takeuchi} the $\Q$-vector space $\Q[\rho]$ spanned by the image of $\rho$ in the space $M_2(\R)$ of $2\times 2$ real matrices would form a quaternion algebra over $\Q$. Moreover, the submodule $\Z[\rho]\subset\Q[\rho]$ generated by the image of $\rho$ over $\Z$ is then an order of $\Q[\rho]$. (We shall refer to \cite{voight} (especially Chapter 22, Section 1, \emph{loc.cit.}) for generalities on quaternion orders and ternary quadratic forms used below.) Explicitly, in terms of the optimal sequence of generators $(\alpha,\beta,\gamma)$ of $\pi_1\Sigma$, we would have a basis $\{1,i,j,k\}$ of $\Z(\rho)$ given by $i=\rho(\alpha)$, $j=\rho(\beta)$, $k=(ij)^{-1}=\rho(\beta^{-1}\alpha^{-1})$ with
\begin{align*}
&i^2=xi-1 & jk=\bar i= x-i,\\
&j^2=yj-1 & ki=\bar j= y-j,\\
&k^2=zk-1 & ij=\bar k= z-k
\end{align*}
where $(x,y,z)=(\tr\rho(\alpha),\tr\rho(\beta),\tr\rho(\beta^{-1}\alpha^{-1}))$. The ternary quadratic form associated to $\Z[\rho]$ is then
$$X^2+Y^2+Z^2+xYZ+yXZ+zXY$$
whose discriminant is $4+xyz-x^2-y^2-z^2=2-k$. Hence, $2-k$ must be divisible by the discriminant of the quaterion algebra $\Q[\rho]$, where the latter is defined as the product of all places over which $\Q[\rho]$ is ramified. Recall that $\Q[\rho]$ is split at $\infty$, and $\Q[\rho]$ must be ramified at an even number of primes by quadratic reciprocity. Since $2-k\in\{1,2,3\}$ by our hypothesis that $k\in\{-1,0,1\}$, we conclude that the discriminant of $\Q[\rho]$ must be $1$; hence, $\Q[\rho]$ is isomorphic to the split algebra $M_2(\Q)$ and the image of $\rho$ lies in $\SL_2(\Z)$ up to conjugation. But then the image of $\rho$ cannot be cocompact; this contradiction shows us that $T(1;e)\cap X_k(\Z)$ is indeed empty. This discussion shows that any point on $X_k(\Z)$ for $k\in\{0,\pm1\}$ arises from a unitary representation, and we indeed find
\begin{align*}
X_{-1}(\Z)&=\Gamma\cdot\{(1,0,0)\},\\
X_{0}(\Z)&=\Gamma\cdot\{(1,1,0)\},\\
X_1(\Z)&=\emptyset.
\end{align*}
On the other hand, for $k=-2$, by Markoff descent we find that
$$X_k(\Z)=\Gamma\cdot\{(0,0,0),(3,3,3),(-3,-3,3),(-3,3,-3),(3,-3,-3)\}.$$
The compact component of $X_k(\R)$ in this case is just the singleton $\{(0,0,0)\}$. From this, we find that there exists exactly one arithmetic hyperbolic surface of signature $(1,\infty)$ derived from a quaternion algebra, namely, the quotient of $\H^2$ by the lattice
$$L=\left\langle\begin{bmatrix}1 & 1\\ 1 & 2\end{bmatrix},\begin{bmatrix}1 & -1\\ -1 & 2\end{bmatrix}\right\rangle\leq\PSL_2(\Z)$$
which is also called the \emph{modular torus} in the literature. This is found by noting that
$$(3,3,3)=\left(\tr\begin{bmatrix}1 & 1\\ 1 & 2\end{bmatrix},\tr\begin{bmatrix}1 & -1\\ -1 & 2\end{bmatrix},\tr\begin{bmatrix}0 & 1\\ -1 & 3\end{bmatrix}\right).$$
We remark that Takeuchi \cite{takeuchi2} gave a full classification of arithmetic hyperbolic surfaces of signature $(1,e)$. In a related context, Bowditch-Maclachlan-Reid \cite{bmr} classified arithmetic once-punctured torus bundles by considering imaginary quadratic integral points on $X_k(\Sigma)$ for $(g,n,k)=(1,1,-2)$.

\subsection{Faithful representations}\label{sect:5.2}
Let $\Sigma$ be a surface of type $(g,n)$ with $3g+n-3>0$ and $n\geq1$, and let $c_1,\dots,c_n$ be the boundary curves of $\Sigma$. Fix $k\in\A^n(\Z)$. Inspired by a suggestion of Curtis McMullen, we give another proof of the finitude of nondegenerate integral mapping class group orbits in $X_k(\Sigma,\Z)$ corresponding to \emph{faithful} representations $\pi_1\Sigma\to\SL_2(\R)$.

\begin{proposition}
\label{faithful}
The set of nondegenerate points of $X_k(\Sigma,\Z)$ corresponding to faithful representations consists of finitely many mapping class group orbits.
\end{proposition}

\begin{proof}
Let $\rho:\pi_1\Sigma\to\SL_2(\R)$ be a faithful representation with integral character having boundary traces $k$. The image of $\rho$ is a torsionfree discrete non-cocompact subgroup of $\SL_2(\R)$. Let $S_\rho$ denote the Nielsen core of the quotient hyperbolic surface $\H^2/\Img\rho$. For each $i=1,\dots,n$, the conjugacy class $d_i=\rho(c_i)$ in $\Img\rho$ gives rise to a unique closed (not necessarily simple) geodesic on $S_\rho$, also denoted $d_i$ for simplicity. Let $T_\rho\subset S_\rho$ denote a subsurface obtained by taking the union of sufficiently small closed tubular neighbourhoods of the geodesics $d_i$. Note that we may choose the neighbourhoods so that the lengths of the boundary curves of $T_\rho$ are bounded in terms of $k$.

Suppose first that all of the boundary curves of $S_\rho$ are each isotopic to a boundary curve of $T_\rho$, so that their lengths are bounded in terms $k$. Since the traces of elements in $\Img\rho$ are in $\Z$, it follows by an extension of McKean's theorem on finitude of length isospectral hyperbolic surfaces (which may be proved via mapping class group dynamics on Teichm\"uller locus by arguing as in the proof of Theorem \ref{theorem4} and its Corollary \ref{finiteness}) that there are only finitely many possibilities for $S_\rho$ up to isometry, or $\Img\rho$ up to conjugacy in $\SL_2(\R)$. Moreover, for each $i\in\{1,\dots,n\}$ there are only finitely many possibilities for the conjugacy class $d_i$ in $\Img\rho$. Moreover, given $\rho$ (and hence the data of $(\Img\rho,d_1,\dots,d_n)$ as above), by the Dehn-Nielsen-Baer theorem any isomorphism
$$\rho':\pi_1\Sigma\xrightarrow{\sim}\Img\rho$$
sending the conjugacy class $c_i$ to $d_i$ is $\Gamma(\Sigma)$-equivalent to $\rho$. This shows that there are at most finitely many mapping class group orbits in $X_k(\Sigma,\Z)$ corresponding to those $\rho$ such that each boundary curve of $S_\rho$ is isotopic to a boundary curve of $T_\rho$.

To complete our proof, it therefore suffices to show that, if there is a boundary curve of $S_\rho$ which is not isotopic to a boundary curve of $T_\rho$, then $\rho$ is degenerate. Suppose $e\subset\del S_\rho$ is a boundary curve with such property. By our hypothesis, there is an embedded subsurface $\Sigma_1\subset S_\rho$ of type $(0,3)$ which is disjoint from $T_\rho$ and contains $e$ as a boundary curve. Let us label the boundary curves $e_1=e$, $e_2$, and $e_3$. Let us choose a base point of $S_\rho$ lying in $\Sigma_1$, and let $(\gamma_1,\gamma_2,\gamma_3)$ be an optimal sequence of generators for $\pi_1\Sigma_1$ with $\gamma_i$ corresponding to $e_i$. Let us choose a marking of the hyperbolic structure on $S_\rho$, and consider the associated representation
$$\sigma:\pi_1S_\rho\simeq\Img\rho\hookrightarrow\SL_2(\R).$$
Our idea is to construct a suitable deformation $\sigma_t:\pi_1S_\rho\to\SL_2(\C)$ of $\sigma=\sigma_0$ which keeps the traces $\tr\sigma_t(d_i)$ constant for $i=1,\dots,n$ (cf.~proof of \cite[Theorem 3.2]{whang4}). This is done as follows. Note first that $\sigma|\Sigma_1$ is irreducible, and hence we may assume, up to global conjugation in $\SL_2(\C)$, that
$$\sigma(\gamma_2)=\begin{bmatrix}\mu &0\\0 & \mu^{-1}\end{bmatrix}$$
for some $\mu\in\C\setminus\{0,\pm1\}$ while $\sigma(\gamma_1)$ is neither upper triangular nor lower triangular. Consider the family of representations
$$\sigma_t':\pi_1\Sigma_1\to\SL_2(\C)$$
given by
$$\sigma_t'(\gamma_1)=\sigma(\gamma_1)\begin{bmatrix}1 & -t\\ 0 & 1\end{bmatrix},\quad \sigma_t'(\gamma_2)=\begin{bmatrix}1 & t\\ 0 &1\end{bmatrix}\sigma'(\gamma_2)$$
for $t\in\C$. Note that $\sigma_0'=\sigma|\Sigma_1$, and for each $t\in\C$ we have
$$\sigma_t'(\gamma_3)=\sigma(\gamma_3)\quad\text{and}\quad\tr\sigma_t'(\gamma_2)=\tr\sigma(\gamma).$$
Denoting by $\Sigma_2$ the connected component of $S_\rho|e_2$ containing the boundary curve $e$ (which we take to mean $S_\rho$ itself if $e_2$ is a boundary curve of $S_\rho$), we note that $\sigma_t'$ extends easily to a family $\sigma_t'':\pi_1\Sigma_2\to\SL_2(\C)$ of representations such that $\sigma_t''|_{\Sigma_1}=\sigma_t'$ and $\sigma_t''|_{\Sigma_1^{\circ}}=\sigma|_{\Sigma_1^{\circ}}$ where $\Sigma_1^{\circ}$ is the complement of $\Sigma_1$ in $\Sigma_2$ (possibly empty). We next extend $\sigma_t''$ to a family $\sigma_t:\pi_1 S_\rho\to\SL_2(\C)$ by considering three cases.
\begin{itemize}
	\item Suppose $a_2$ is a boundary curve of $S_\rho$. Set $\sigma_t=\sigma_t''$ and we are done.
	\item Suppose $a_2$ is a separating curve of $S_\rho$, and let $\Sigma_2^\circ$ be the component of $S_\rho|a_2$ complementary to $\Sigma_2$. We define $\sigma_t$ by requiring
	$$\sigma_t\Sigma_2=\sigma_t'',\quad \sigma_t|{\Sigma_2^{\circ}}=\begin{bmatrix}1 & \frac{-t\mu^{-1}}{\mu-\mu^{-1}}\\0 & 1\end{bmatrix}\sigma|_{\Sigma_2^\circ}\begin{bmatrix}1 & \frac{-t\mu^{-1}}{\mu-\mu^{-1}}\\0 & 1\end{bmatrix}^{-1}.$$
	By construction, we have
	$$\sigma_t(\gamma_2)=\begin{bmatrix}1 & \frac{-t\mu^{-1}}{\mu-\mu^{-1}}\\0 & 1\end{bmatrix}\sigma(\gamma_2)\begin{bmatrix}1 & \frac{-t\mu^{-1}}{\mu-\mu^{-1}}\\0 & 1\end{bmatrix}^{-1}$$
	so that $\rho_t$ is well defined by our discussion in Section \ref{sect:2.2.3}.
	\item Suppose $a_2$ is a nonseparating curve of $S_\rho$. Let $\beta$ be a simple based loop on $S_\rho$ which intersects $\gamma_2$ exactly once transversally, oriented as in the part of Section \ref{sect:2.2.3} (with the pair $(\gamma_2,\beta)$ playing the role of $(\alpha,\beta)$ there). Let $\gamma_2'$ be the based loop on $\Sigma_2$ (like $\alpha_2'$ in Section \ref{sect:2.2.3}) whose image in $S_\rho$ lies in the homotopy class of $\delta^{-1}\gamma_2\delta$. We define the representation $\rho_t$ by specifying the pair $(\sigma_t|_{\Sigma_2},\sigma_t(\delta))$ as in the discussion of Section \ref{sect:2.2.3}, where
	$$\sigma_t|_{\Sigma_2}=\rho_t',\quad \sigma_t(\delta)=\begin{bmatrix}1 & \frac{-t\mu^{-1}}{\mu-\mu^{-1}}\\ 0 & 1\end{bmatrix}\sigma(\delta).$$
	We have $\sigma_t(\delta^{-1})\sigma_t(\gamma_2)\sigma_t(\delta)=\sigma(\delta^{-1})\sigma(\gamma_2)\sigma(\delta)=\sigma(\gamma_2')=\sigma_t(\gamma_2')$ by our construction, so $\sigma_t$ is well defined.
\end{itemize}
We thus have a family of representations $\sigma_t:\pi_1S_\rho\to\SL_2(\C)$ such that $\sigma_0=\sigma$, and moreover note that $\tr\sigma_t(d_i)$ is constant for all $t\in\C$ by construction since $\Sigma_1$ is disjoint from $T_\sigma$. Considering the assignment sending $t$ to the composition
$$\rho_t=\sigma_t\circ\rho:\pi_1\Sigma\xrightarrow{\sim}\pi_1 S_\rho\to\SL_2(\C),$$
we obtain a map $\A^1\to X_k(\Sigma)$ which is easily seen to be polynomial, and contains the class of $\rho$ in its image. Moreover, since $\tr\sigma_t(e)$ is nonconstant, it follows that $\rho_t$ is nonconstant as well, showing that $\rho$ is degenerate as desired.
\end{proof}

\end{document}